\documentclass[12pt]{article}
\usepackage[margin = .991in]{geometry}
\usepackage{amsthm,amssymb,amsmath,graphicx,cite, color,mathrsfs,booktabs}
\usepackage{algpseudocode,algorithm,algorithmicx}

\newcommand{\Xomit}[1]{}

\newcommand{\ip}[2]{\langle #1 , #2 \rangle}    
\newcommand{\dmin}{\displaystyle\min}
\newcommand{\dmax}{\displaystyle\max}
\newcommand{\dsum}{\displaystyle\sum}
\newcommand{\R}{{\mathbb R}}
\renewcommand{\S}{{\mathfrak S}}
\newcommand{\sJ}{\mathcal S}

\newcommand{\B}{{\mathbb B}}

\newcommand{\affine}{{\mathrm{affine}}}

\newcommand{\dom}{{\mathrm{dom}}}
\newcommand{\dist}{{\mathrm{dist}}}

\newcommand{\graph}{{\mathrm{graph}}}
\newcommand{\ext}{{\mathrm{ext}}}
\newcommand{\vertiii}[1]{\| #1 \|}
\newcommand{\vertii}[1]{{\vert\kern-0.25ex\vert\kern-0.25ex\vert #1     \vert\kern-0.25ex\vert\kern-0.25ex\vert}}
\usepackage{enumitem}

\newcommand{\1}{{\mathbf 1}}
\newcommand{\ri}{{\mathrm{relint}}}

\DeclareMathOperator{\Image}{Im}

\DeclareMathOperator{\st}{s.t.}

\newtheorem{lemma}{Lemma}
\newtheorem{theorem}{Theorem}
\newtheorem{proposition}{Proposition}

\newtheorem{corollary}{Corollary}

\newtheorem{example}{Example}
\def\transp{^{\text{\sf T}}}
\newcommand{\T}{\mathcal T}

\newcommand{\F}{\mathcal F}
\renewcommand{\H}{\mathcal H}
\newcommand{\G}{\mathcal G}
\newcommand{\I}{\mathcal I}
\newcommand{\Id}{\text{\sf{I}}}
\newcommand{\matr}[1]{\begin{bmatrix} #1 \end{bmatrix}}    

\title{New characterizations of Hoffman constants for systems of linear constraints }
\author{Javier Pe\~na\thanks{Tepper School of Business,
Carnegie Mellon University, USA, {\tt jfp@andrew.cmu.edu}} \and Juan Vera\thanks{Department of Econometrics and
Operations Research,
Tilburg University, The Netherlands, {\tt j.c.veralizcano@uvt.nl}}
\and Luis F. Zuluaga\thanks{Department of Industrial and Systems Engineering, Lehigh University, USA, {\tt
luis.zuluaga@lehigh.edu}}}

\begin{document}

\maketitle

\begin{abstract}
We give a characterization of the Hoffman constant of a system of linear constraints in $\R^n$  {\em relative} to a {\em reference polyhedron} $R\subseteq\R^n$.  The reference polyhedron $R$ represents
constraints that are easy to satisfy such as box constraints.
In the special case $R = \R^n$, we obtain a novel characterization of the classical Hoffman constant.

More precisely,  suppose $R\subseteq \mathbb{R}^n$ is  a reference polyhedron, $A\in \R^{m\times n},$ and $A(R):=\{Ax: x\in R\}$. We characterize the
sharpest constant $H(A\vert R)$ such that for all $b \in A(R) + \R^m_+$ and $u\in R$
\[
\dist(u, P_{A}(b)\cap R) \le H(A\vert R) \cdot \|(Au-b)_+\|,
\]
where $P_A(b) = \{x\in \R^n:Ax\le b\}$.
Our characterization is stated in terms of the largest of a canonical collection of easily computable Hoffman
constants.  Our characterization in turn suggests new algorithmic procedures to compute  Hoffman constants.

\end{abstract}
\section{Introduction}
\label{sec.intro}

A classical result of Hoffman~\cite{Hoff52} shows that the distance from a point $u \in \R^n$ to a non-empty
polyhedron $P_{A}(b):=\{x \in \R^n: Ax \le b\}$ can be bounded above in terms of the size of the {\em residual} vector
$(Au-b)_+ := \max(0,Au-b)$.  More precisely, for $A \in \R^{m\times n}$ there exists a {\em Hoffman constant} $H(A)$
that depends only on $A$ such that for all $b\in A(\R^n) + \R^m_+$  and all $u\in \R^n$,
\begin{equation}
\label{eq:simple_erro}
\dist(u,P_{A}(b)) \le H(A) \cdot \|(Au-b)_+\|.
\end{equation}
Here $A(\R^n):=\{Ax: x\in \R^n\}$ and $\dist(u,P_{A}(b)) := \min\{\|u-x\|: x\in P_{A}(b)\}$. 
For convenience, we will make the following slight abuse of notation throughout the paper.  We will write $\|\cdot\|$ to denote both the norm in $\R^n$ and the norm in $\R^m$.  The specific norm will always be evident from the context.
The bound~\eqref{eq:simple_erro}  is a type of {\em error bound} for the system of inequalities $Ax \le b$, that is,
an inequality bounding the distance from a point  $u\in \R^n$ to a nonempty {\em solution set} in terms of a measure
of the {\em error} or {\em residual} of the point $u$ for the constraints defining the solution set.

We consider the following more general {\em relative} version of~\eqref{eq:simple_erro}.
Suppose $R\subseteq \R^n$ is a nonempty {\em reference polyhedron} and $A \in \R^{m\times n}$.  The reference
polyhedron $R$ represents some constraints that are easy to satisfy such as box constraints.   Let $A(R):=\{Ax: x\in R\}$.  We give a characterization of the sharpest {\em relative Hoffman constant} $H(A\vert R)$ that depends only on $(A,R)$ such that for all $b\in A(R) +
\R^m_+$ and all $u\in R$,
\begin{equation}
\label{eq:general_erro}
\dist(u,P_{A}(b)\cap R) \le H(A\vert R) \cdot \|(Au-b)_+\|.
\end{equation}
In the special case when $R = \R^n$ we have $H(A\vert R) = H(A)$ and obtain the following novel characterization of
$H(A)$:
\begin{equation}\label{eq.new.charac}
H(A) = \max_{J \in \sJ(A) }  \frac{1}{\dmin_{v\in \R^J_+, \|v\|^*=1}\|A_J\transp v\|^*},
\end{equation}
where $\sJ(A)$ is the collection of subsets $J\subseteq \{1,\dots,m\}$ such that $A_J(\R^n) + \R^J_+ = \R^J$, $A_J$ is the submatrix of $A$ defined by the rows indexed by $J$, and $\|\cdot\|^*$ denotes the dual norm of $\|\cdot\|$.  Observe that $A_J(\R^n) + \R^J_+ = \R^J$ if and only if $A_Jx < 0$ is feasible.

Hoffman bounds of the classical form~\eqref{eq:simple_erro}, the relative form~\eqref{eq:general_erro}, and more
general error bounds play a fundamental role in mathematical programming~\cite{Nguy17,Pang97,ZhouS17}.  In particular,
these kinds of Hoffman bounds as well as other related error bounds play a central role in establishing convergence
properties of a variety of modern convex optimization
algorithms~\cite{BeckS15,Garb18,GutmP19,LacoJ15,LeveL10,LuoT93,NecoNG18,PenaR16,WangL14}.  Hoffman bounds are also
used to measure the
optimality  and  feasibility  of  a  point  generated  by  rounding
an  optimal  point  of  the continuous  relaxation  of  a  mixed-integer  linear or quadratic
optimization problem~\cite{stein2016,granot1990}. Furthermore, Hoffman bounds are used in
sensitivity analysis~\cite{jourani2000}, and to design solution methods for non-convex quadratic
programs~\cite{xia2015}.

The relative format~\eqref{eq:general_erro} that includes a reference polyhedron arises naturally in various contexts.
For instance, it usually occurs when there are {\em box constraints} of the form $\ell \le x \le u$ as these constraints are generally easy to satisfy.   Our interest in
characterizing the Hoffman constant in the more general relative case that includes a reference polyhedron is
motivated by the recent articles~\cite{BeckS15,LacoJ15,Garb18,GutmP19,PenaR16,xia2015}.  In each  of these articles,
relative Hoffman constants for systems of linear constraints for suitable reference polyhedra play a central role in
establishing key properties of modern optimization algorithms.  In particular, the {\em facial distance} or {\em
pyramidal width} introduced in~\cite{LacoJ15,PenaR16} is precisely a relative Hoffman constant with the standard simplex as reference polyhedron.

The paper makes the following main contributions.  First, we develop novel characterizations of Hoffman constants for
systems of linear inequalities.  The characterization is stated as the largest of a canonical collection of easily computable Hoffman constants.    Our
characterization applies to the general case that includes both linear inequalities and linear equations, and a reference polyhedron representing constraints that
are easy to satisfy.
As a special case we obtain the new characterization~\eqref{eq.new.charac} for the classical Hoffman constant $H(A)$.  The Hoffman constant and our characterization of it also extend to the broader context of polyhedral set-valued mappings.

Second, throughout the paper we highlight the interesting and natural but mostly overlooked connection between the
Hoffman constant and Renegar's distance to ill-posedness~\cite{Rene95a,Rene95b}, which is a cornerstone of condition
measures in continuous optimization.  Our connection is along the lines of some developments by  Ramdas and Pe\~na~\cite{RamdP16}.  In particular, we detail the tight connection between the reciprocal of the
Hoffman constant $1/H(A)$ and Renegar's distance to ill-posedness for the system of linear inequalities $Ax < 0$.  We
also discuss similar interesting connections for other Hoffman constants.

Third, we leverage our characterizations of the relative Hoffman constants $H(A)$ and $H(A\vert R)$ to develop novel
algorithmic approaches to compute or estimate Hoffman constants.  We should note that the exact or even approximate
computation of the Hoffman constant is a notoriously difficult and largely unexplored computational challenge.  
The characterization~\eqref{eq.new.charac} suggests that while any $J\in \sJ(A)$ yields a  lower bound on $H(A)$, the typically more interesting task of computing a reasonable upper bound on $H(A)$ is far more challenging since the upper bound must hold for all $J\in \sJ(A)$.
Aside from the algorithmic procedures proposed in this paper, there appears to be only one other documented method
to compute the Hoffman constant, namely Algorithm {\tt ALG 2} proposed by Klatte and Thiere in~\cite{KlatT95}.  This
algorithm is based on the following popular characterization of $H(A)$ from~\cite{guler1995,KlatT95,WangL14}
\begin{equation}\label{eq.popular}
H(A) = \max_{J\subseteq \{1,\dots,m\}\atop A_J \text{ full row rank} }  \frac{1}{\dmin_{v\in \R^J_+,
\|v\|^*=1}\|A_J\transp v\|^*}.
\end{equation}
The characterization~\eqref{eq.popular} is often alluded to in the optimization literature as an expression for
computing $H(A)$.  Indeed, Algorithm {\tt ALG 2} of Klatte and Thiere in~\cite{KlatT95} is based on this
characterization.  It performs an exhaustive search over all $J\subseteq \{1,\dots,m\}$ to evaluate~\eqref{eq.popular}
which evidently is viable only for very small values of $m$.  A main limitation of~\eqref{eq.popular} is that it does
not take advantage of any structural properties of $A$.  As we discuss in Section~\ref{sec.algo}, it is possible to
compute $H(A)$ via a variant of~\eqref{eq.new.charac} that takes the maximum over a potentially much smaller
collection of subsets $\F\subseteq \sJ(A)$.  In the most favorable case when $A(\R^n) + \R^m_+ = \R^m$ we can take $\F =
\{\{1,\dots,m\}\}$ and thus $$H(A)=\frac{1}{\dmin_{v\in \R^m_+, \|v\|^*=1}\|A\transp v\|^*},$$ which can be computed
via a single and fairly tractable convex optimization problem for suitable choices of norms.

\medskip

The paper is entirely self-contained and relies only on standard convex optimization techniques.
Our results are related to a number of previous developments in the rich literature on error
bounds~\cite{AzeC02,burke1996,guler1995,Li93,MangS87,robinson1973,VanNT09,Zali03} and on condition measures for
continuous optimization~\cite{BurgC13,EpelF02,FreuV99a,FreuV03,Freu04,Lewi99,Pena00,Pena03,Rene95a,Rene95b}.  In
particular, the expressions for the Hoffman constants in Proposition~\ref{prop.Hoffman} and
Proposition~\ref{prop.Hoffman.alt} have appeared, albeit in slightly different form or under more restrictive assumptions,
in the work of  Klatte and Thiere~\cite{KlatT95}, Li~\cite{Li93},   Robinson~\cite{robinson1973}, and Wang and
Lin~\cite{WangL14}.  More precisely, Klatte and Thiere~\cite{KlatT95} state and prove a version of
Proposition~\ref{prop.Hoffman} under  the more restrictive assumption that $\R^n$ is endowed with the $\ell_2$ norm.
Li~\cite{Li93},   Robinson~\cite{robinson1973}, and Wang and Lin~\cite{WangL14} give characterizations of Hoffman
constants that are equivalent to Proposition~\ref{prop.Hoffman} and Proposition~\ref{prop.Hoffman.alt} but where the
maximum is taken over a different, and typically much larger, collection of index sets.
As we detail in Section~\ref{sec.special}, our characterization  for $H(A)$ in Proposition~\ref{prop.Hoffman} and Proposition~\ref{prop.Hoffman.alt} implies~\eqref{eq.popular} and can
readily be seen to be at least as sharp as some bounds on~$H(A)$ previously derived by G\"uler et al.~\cite{guler1995}, Burke
and Tseng~\cite{burke1996}, and Zhang~\cite{Zhan00}.  We also note that weaker versions of Theorem~\ref{thm.main.gral} can be obtained from
results on error bounds in Asplund spaces as those developed in the article by Van Ngai and Th{\'e}ra~\cite{VanNT09}.
Our goal to characterize a relative version of Hoffman constants that accounts for the presence of a reference
polyhedron is partly inspired by the concepts of relative smoothness, relative strong convexity, and relative continuity recently developed and used by Bauschke et al~\cite{BausBT16}, Lu~\cite{Lu17}, Lu et al~\cite{LuFN18}, and
Teboulle~\cite{Tebo18}.
Our characterization of the Hoffman constants is in the spirit of and draws on the seminal work by
Renegar~\cite{Rene95a,Rene95b} as well as related work by Freund and Vera~\cite{FreuV99a,FreuV03},
Pe\~na~\cite{Pena00,Pena03} and Lewis~\cite{Lewi99,Lewi05}.

The contents of the paper are organized as follows.  Section~\ref{sec.special} presents our main developments.  We
give a novel characterization of the classical Hoffman constant $H(A)$ (see Proposition~\ref{prop.Hoffman},
Proposition~\ref{prop.Hoffman.alt}, and Corollary~\ref{corol.sets}).  We also extend this characterization to the more
general case that includes linear inequalities, linear equations, and a reference polyhedron (see
Proposition~\ref{prop.Hoffman.gral.2}, Proposition~\ref{prop.Hoffman.gral.alt.2}, and Corollary~\ref{corol.ineq.2}).
Section~\ref{sec.algo}  describes several algorithmic procedures to compute Hoffman constants.
Section~\ref{sec.proof}, the most technical section of the paper, presents developments similar to those in Section~\ref{sec.special} but in the broader context
of polyhedral set-valued mappings (see Theorem~\ref{thm.main.gral}, Theorem~\ref{thm.main.gral.alt}, and
Corollary~\ref{corol.phi.sets}).  Finally, Section~\ref{sec.proof.props} details how the results in
Section~\ref{sec.special} follow from those in Section~\ref{sec.proof}.

Throughout the paper whenever we work with an Euclidean space $\R^d$, we will assume that it is endowed with a (non-necessarily Euclidean) norm
$\|\cdot\|$ and inner product $\ip{\cdot}{\cdot}$.  We will often use the dual norm  $\|\cdot\|^*$ of $\|\cdot\|$
defined as follows
\[
\|u\|^*:=\max_{\|x\|\le 1} \ip{u}{x}.
\]
Unless we explicitly state otherwise, our results  apply to arbitrary norms.

We will also rely on the following notation.  
Given a polyhedron $Q\subseteq \R^d$ let $\T(Q):=\{T_Q(u): u \in Q\}$
where $T_Q(u)$ denotes the {\em tangent cone} to $Q$ at $u\in Q$, that is,
\[
T_Q(u) = \{d \in \R^d: u + td \in Q \; \text{ for some } t > 0\}.
\]
Observe that since $Q$ is assumed to be a polyhedron, the collection of tangent cones $\mathcal T(Q)$ is finite.  Given a convex cone $K\subseteq\R^n$ we let $K^*$ denote its dual cone, that is,
\[
K^*:=\{u\in \R^n: \ip{u}{x} \ge 0 \text{ for all } x\in K\}.
\]
Throughout the paper we will write $[m]$ as shorthand for $\{1,\dots,m\}$.

\section{Hoffman constants for systems of linear constraints}
\label{sec.special}

This section describes a characterization for the Hoffman constant $H(A \vert R)$  in~\eqref{eq:general_erro} for systems of
linear inequalities
\[
\begin{array}{l}
Ax \le b\\
x\in R.
\end{array}
\]
We subsequently consider analogous Hoffman constants for systems of linear  inequalities and linear equations
\[
\begin{array}{l}
Ax \le b \\
Cx = d\\
x\in R.
\end{array}
\]
Although the latter case with inequalities and equations subsumes the former case, for exposition purposes we discuss separately the case with inequalities only.  Furthermore, we start with the special case $R=\R^n$.  
The notation and main ideas in this special case are simpler and easier to grasp. 
In particular, the crux of the characterization of $H(A)$  based on a canonical collection of submatrices of $A$ is more apparent.

We defer the proofs of the propositions in this section to Section~\ref{sec.proof.props}, where we show that they
follow from more general results for polyhedral set-valued mappings detailed in Section~\ref{sec.proof}.

\subsection{The case of inequalities only}
\label{sec.ineq}
Proposition~\ref{prop.Hoffman} below gives a  characterization of the {\em sharpest} Hoffman constant $H(A)$ such
that \eqref{eq:simple_erro} holds.
The characterization is stated in terms of a canonical collection of submatrices of $A$.

\medskip

Let  $A\in \R^{m\times n}$.  We next consider systems of linear inequalities of the form
\[
Ax \le b.
\]
Let $\sJ(A):= \{J \subseteq [m]: A_J(\R^n) + \R^J_+ = \R^J\}$ and
\begin{equation}\label{eq.Hoffman.matr}
H(A):=\dmax_{J\in \sJ(A)} H_J(A),
\end{equation}
where
\[
H_J(A):= \dmax_{y\in \R^m \atop \|y\|\le 1} \dmin_{x\in \R^n \atop A_Jx \le y_J} \|x\|
\]
for each $J\in \sJ(A)$.  By convention $H_J(A) := 0$ if $J = \emptyset$.

\medskip

Throughout the sequel, for $A\in \R^{m\times n}$ and $b\in \R^m$ we let $P_A(b)$ denote the polyhedron $\{x\in
\R^n: Ax \le b\}$.

\begin{proposition}\label{prop.Hoffman}
Let  $A\in \R^{m\times n}$.  Then for all $b \in A(\R^n) + \R^m_+$ and all $u\in \R^n$
\begin{equation}\label{eq.Hoffman.bound.matr}
\dist(u,P_{A}(b)) \le H(A)\cdot \dist(b-Au,\R^m_+) \le H(A)\cdot\|(Au-b)_+\|.
\end{equation}
Furthermore, the first bound in~\eqref{eq.Hoffman.bound.matr} is tight: If $H(A) > 0$ then there exist $b\in
A(\R^n)+\R^m_+$  and $u \in\R^n$ such that
\[
\dist(u,P_{A}(b)) = H(A)\cdot \dist(b-Au,\R^m_+) > 0.
\]
\end{proposition}

The following result complements Proposition~\ref{prop.Hoffman} and yields an alternative expression for $H(A)$.

\begin{proposition}\label{prop.Hoffman.alt} Let  $A\in \R^{m\times n}$. Then for all $J \in \sJ(A)$
\begin{equation}\label{eq.HJ}
H_J(A) = \dmax_{y\in \R^m\atop \|y\|\le 1}\dmin_{x\in \R^n \atop A_Jx \le y_J} \|x\| = \frac{1}{\dmin_{v \in \R^J_+,
\| v\|^*= 1} \|A_J\transp v\|^*},
\end{equation}
with the convention that the last denominator is  $+\infty$ and thus $H_J(A) = 0$ when $J=\emptyset.$ In particular,
\begin{equation}\label{eq.Hoffman.alt}
H(A) = \dmax_{J\in \sJ(A)} \frac{1}{\dmin_{v \in \R^J_+, \; \|v\|^* =1 } \|A_J\transp v\|^*}.
\end{equation}
\end{proposition}
Proposition~\ref{prop.Hoffman.alt} implies that
\begin{align*}
H(A) &= \max_{J \in \sJ(A)} \max\{\|v\|^*: v \in\R^J_+, \|A_J\transp v\|^* \le 1\}\\
& = \max_{J \in \sJ(A)} \max\{\|\tilde v\|^*: \tilde v \in \ext\{v \in\R^J_+, \|A_J\transp v\|^* \le 1\}\},
\end{align*}
where $\ext(C)$ denotes the set of extreme points of a closed convex set $C$.  Thus the following bound on $H(A)$ previously established
in~\cite{burke1996,guler1995} readily follows:
\begin{align*}
H(A) &= \max_{J \in \sJ(A)} \max\{\|\tilde v\|^*: \tilde v \in \ext\{v \in\R^J_+, \|A_J\transp v\|^* \le 1\}\}
\\
&\le \max\{\|\tilde v\|^*: \tilde v \in \ext\{v \in\R^m_+, \|A\transp v\|^* \le 1\}\}.
\end{align*}

Furthermore, observe that if $J\in\sJ(A)$ and $\tilde v \in \ext\{v \in\R^J_+, \|A_J\transp v\|^* \le 1\}$ then $A_{J'}$ must have full row rank for $J' :=\{i:\tilde v_i > 0 \}\subseteq J.$ Therefore Proposition~\ref{prop.Hoffman.alt} also implies that 
\[
H(A) =\max_{J\subseteq \{1,\dots,m\}\atop A_J \text{ full row rank} } \max \{\| v\|^*:v\in \R^J_+, \|A_J\transp v\|^*\le 1 \}
 = \max_{J\subseteq \{1,\dots,m\}\atop A_J \text{ full row rank} }  \frac{1}{\dmin_{v\in \R^J_+,
\|v\|^*=1}\|A_J\transp v\|^*},
\] 
which is precisely the characterization~\eqref{eq.popular} of $H(A)$.
In addition, Proposition~\ref{prop.Hoffman.alt} 
implies the following bound on $H(A)$ in terms of the $\chi(A)$ condition measure~\cite{Stew89,Todd90,VavaY96} established in~\cite{Zhan00} for the special case when $A\in\R^{m\times n}$ is full column rank and both $\R^m$ and $\R^n$ are endowed with Euclidean norms:
\begin{align*}
H(A) &= \max_{J\subseteq \{1,\dots,m\}\atop A_J \text{ full row rank} } \max\{\|v\|: v \in\R^J_+, \|A_J\transp v\| \le 1\}\\
&= \max_{J\subseteq [m], |J| = n\atop A_J \text{ non-singular} } \max\{\|v\|: v \in\R^J_+, \|A_J\transp v\| \le 1\}\\
&\le \max_{J \subseteq [m], |J|=n\atop A_J \text{non-singular}} \max\{\|v\|: v \in\R^J, \|A_J\transp v\| \le 1\} \\
&= \max_{J \subseteq [m], |J|=n\atop A_J \text{non-singular}} \|A_J^{-1}\|\\
&=\chi(A).
\end{align*}
The last step follows from\cite[Prop. 3.7]{Zhan00}.

Observe that the above inequality could be fairly loose.  For example if
$
A = \matr{1 & -\epsilon\\ 1 & \;\;\epsilon}
$ 
for some small $\epsilon >0$, then $H(A) = 1+{\mathcal O}(\epsilon)$ whereas $\chi(A) = \|A^{-1}\| = \Omega(1/\epsilon)$.

\medskip

Proposition~\ref{prop.Hoffman.alt} also implies that the Hoffman constant $H(A)$ can be computed by maximizing over a
potentially much smaller collection $\F \subseteq \sJ(A)$ as stated in equation~\eqref{eq.Hoffman.alt.short} below.  The expressions~\eqref{eq.Hoffman.alt} and~\eqref{eq.Hoffman.alt.short} for $H(A)$  are at the heart of one of the algorithmic procedures for
computing $H(A)$ that we discuss in Section~\ref{sec.algo}.

\begin{corollary}\label{corol.sets}
Let $A \in \R^{m\times n}$.

\begin{description}
\item[(a)] If $A(\R^n) + \R^m_+ = \R^m$ then
\begin{equation}\label{eq.HA}
H(A) = \frac{1}{\dmin_{v \in \R^m_+, \; \|v\|^* =1 } \|A\transp v\|^*}.
\end{equation}
\item[(b)]
  Suppose $\F\subseteq \sJ(A)$ and $\I \subseteq 2^{[m]}\setminus \sJ(A)$
are such that for all $J \subseteq [m]$ either  $J\subseteq F$ for some $F \in \F,$
or $I\subseteq J$ for some $I \in \I.$   Then
\begin{equation}\label{eq.Hoffman.alt.short}
H(A) =  \dmax_{J \in \F}
H_J(A) =  \dmax_{J \in \F}  \frac{1}{\dmin_{v \in \R^J_+, \| v\|^*= 1} \|A_J\transp v\|^*}.
\end{equation}
\end{description}
\end{corollary}
\begin{proof}
\begin{description}
\item[(a)] This follows from part (b) applied to $\F = \{[m]\}$ and $\I = \emptyset$.
\item[(b)] The conditions on $\F$ and $\I$ imply that for all $J\in\sJ(A)$ there exists $F\in \F$ such that
    $J\subseteq F$.  The latter condition implies $H_J(A) \le H_F(A)$.  Therefore
\[
H(A) = \dmax_{J \in \sJ(A)} H_J(A)
= \dmax_{J \in \F} H_J(A)
= \dmax_{J \in \F}  \frac{1}{\dmin_{v \in \R^J_+, \| v\|^*= 1} \|A_J\transp v\|^*}.
\]
\end{description}
\end{proof}

The identity~\eqref{eq.HA} in Corollary~\ref{corol.sets} has the following geometric interpretation.  By Gordan's
Theorem, $Ax<0$ has a solution if and only if $A\transp v = 0, \, v \ge 0, \, v\ne 0$ does not.  Equivalently,
$ A(\R^n) + \R^m_+ = \R^m$ if and only if $0 \not\in \{A\transp v: v\ge 0, \, \|v\|^* = 1\}$.
When this is the case, the quantity $1/H(A)$ is precisely the distance (in the dual norm $\|\cdot\|^*$) from the
origin to $\{A\transp v: v\ge 0, \, \|v\|^* = 1\}$.  The latter quantity in turn equals the {\em distance to
non-surjectivity} of the mapping $x \mapsto Ax + \R^m_+$, that is, the norm of the smallest perturbation matrix
$\Delta A \in \R^{m\times n}$ such that $(A+\Delta A)(\R^n) + \R^m_+\ne \R^m$ as it is detailed in~\cite{Lewi99}.  This
distance to non-surjectivity is the same as Renegar's {\em distance to ill-posedness} of the system of linear
inequalities $Ax < 0$ defined by $A$. The more general identity~\eqref{eq.Hoffman.alt} in
Proposition~\ref{prop.Hoffman.alt} in turn can be interpreted as follows.  The quantity $1/H(A)$ is the smallest
distance to ill-posedness of the collection of the {\em feasible} systems of linear inequalities of the form $A_J x <
0$ for $J\subseteq[m]$.

 The distance to ill-posedness provides the main building block for Renegar's concept of {\em condition number} for
 convex optimization introduced in the seminal papers~\cite{Rene95a,Rene95b} that has been further extended in~\cite{AmelB12,BurgC13,EpelF02,FreuV99a,FreuV03,Freu04,Pena00,Pena03} among  many other articles.

\bigskip

Proposition~\ref{prop.Hoffman}, Proposition~\ref{prop.Hoffman.alt}, and Corollary~\ref{corol.sets}  extend to the more
general context when there is a reference polyhedron representing some constraints that are easy to satisfy.  More precisely, let
$R\subseteq \R^n$ be a reference polyhedron and $A\in \R^{m\times n}.$  Consider systems of the following form
\[
\begin{array}{r}
Ax \le b\\
x \in R,
\end{array}
\]
where $R$ represents a set of constraints that are easy to satisfy.
It is natural to consider a refinement of the Hoffman constant $H(A)$ that reflects the presence of these easy-to-satisfy
constraints.
To that end, let $\sJ(A \vert R)$ and $H(A\vert R)$ be the extensions of $\sJ(A)$ and $H(A)$ defined as follows.
\[
\sJ(A \vert R) := \{(J,K): J \subseteq [m], K\in \T(R) \text{ and } A_J(K) + \R^J_+ = \R^J\},
\]
and
\[
H(A\vert R) = \dmax_{(J,K)\in \sJ(A\vert R)} H_{J,K}(A),
\]
where
\[
H_{J,K}(A):= \dmax_{y\in \R^m \atop \|y\|\le 1} \dmin_{x\in K \atop A_Jx \le y_J} \|x\|.
\]
Once again, by convention $H_{J,K}(A) := 0$ if $J = \emptyset$.

We have the following analogue of Proposition~\ref{prop.Hoffman}.
\begin{proposition}\label{prop.Hoffman.rel}
Let  $A\in \R^{m\times n}$ and $R\subseteq \R^n$ be a reference polyhedron.  Then for all $b \in A(R) + \R^m_+$ and
$u\in R$
\[
\dist(u,P_A(b)\cap R) \le H(A\vert R)\cdot\dist(b-Au,\R^m_+) \le H(A\vert R)\cdot\|(Au-b)_+\|.
\]
Furthermore the first bound is tight: If $H(A\vert R) > 0$ then there exist $b \in A(R) + \R^m_+$ and $u \in R$ such
that
\[
\dist(u,P_A(b)\cap R) = H(A\vert R)\cdot\dist(b-Au,\R^m_+) > 0.
\]
\end{proposition}

We also have the following analogue of Proposition~\ref{prop.Hoffman.alt} that provides an alternative expression for
$H(A\vert R)$.

\begin{proposition}\label{prop.Hoffman.rel.alt}
Let  $A\in \R^{m\times n}$ and $R\subseteq \R^n$ be a reference polyhedron.  Then for all $(J,K)\in \sJ(A\vert R)$
\begin{equation}\label{eq.HJ.rel}
H_{J,K}(A) = \max_{y\in \R^m \atop \|y\|\le 1} \min_{x\in K\atop A_Jx \le y_J} \|x\| = \frac{1}{\dmin_{v\in \R^J_+,\,
\|v\|^*=1 \atop A_J\transp v-u \in K^*} \|u\|^*}.
\end{equation}
In particular,
\begin{equation}\label{eq.Hoffman.alt.rel}
H(A\vert R)  = \max_{(J,K)\in \sJ(A\vert R)} \frac{1}{\dmin_{v\in \R^J_+,\, \|v\|^*=1 \atop A_J\transp v-u \in K^*}
\|u\|^*}.
\end{equation}
\end{proposition}

We also have the following analogue of Corollary~\ref{corol.sets}.

\begin{corollary}\label{corol.Hoffman.rel}
Let $A \in \R^{m\times n}$ and $R\subseteq\R^n$ be a reference polyhedron.

\begin{description}
\item[(a)] If $R$ is a cone and $A(R) + \R^m_+ = \R^m$ then
\begin{equation}\label{eq.HA.eqns}
H(A\vert R) =
\max_{y\in \R^m \atop \|y\|\le 1} \min_{x\in R\atop Ax \le y} \|x\| = \frac{1}{\dmin_{v\in \R^m_+,\, \|v\|^*=1 \atop
A\transp v-u \in R^*} \|u\|^*}.
\end{equation}
\item[(b)]
  Suppose $\F\subseteq \sJ(A\vert R)$ and $\I \subseteq 2^{[m]}\times\T(R)\setminus \sJ(A\vert R)$
are such that for all $(J,K)\in 2^{[m]}\times\T(R)$ either  $J\subseteq F$ and $T \subseteq K$ for some
$(F,T) \in \F,$
or $I\subseteq J$ and $K\subseteq U$ for some $(I,U) \in \I.$   Then
\[
H(A\vert R) =  \dmax_{(J,K) \in \F} \frac{1}{\dmin_{v\in \R^J_+,\|v\|^*=1\atop A\transp v-u \in K^*} \|u\|^*}.
\]
\end{description}
\end{corollary}

\Xomit{
{\color{blue}
There is a simple but important relationship between the  constants $H(A)$ and $H(A|R)$.  Suppose $A\in \R^{m\times n}$ and $b \in A(\R^n) + \R^m_+$.  Consider the system of inequalities
\[
Ax \le b.
\]
For $u \in \R^n$ Proposition~\ref{prop.Hoffman} yields
\[
\dist(u,P_{A}(b)) \le H(A) \cdot \dist(b-Au,\R^m_+).
\]
Next, suppose \[A = \matr{A^1 \\ A^2}, b = \matr{b^1\\b^2}\] with $A^1\in\R^{k\times n}, b^1\in \R^k$ and $A^2\in\R^{(m-k)\times n}, b^2\in \R^{m-k}$ and $u \in R:=\{x\in \R^n: A^2 x \le b^2\}$.  Then
Proposition~\ref{prop.Hoffman.rel} yields
\[
\dist(u,P_{A}(b)) \le H(A^1|R) \cdot \dist(b^1-A^1u,\R^k_+).
\]
Thus Proposition~\ref{prop.Hoffman.rel} also implies that
\[
H(A^1|R) \le H(A)
\]
provided the norms in $\R^m$ and $\R^{k}$ are {\em compatible} in the following sense
\[
A^2 u \le b^2 \Rightarrow \dist(b^1-A^1u,\R^k_+) = \dist(b-Au,\R^m_+).
\]
Furthermore, $H(A^1|R)$ can be arbitrarily smaller as the following example illustrates. 

\begin{example} 
\label{ex.rel.vs.normal}
Let $m= n = 2$, $k=1$, $A= \matr{1 & 0 \\ 0 & \delta}\in \R^{m\times n}$ with $\delta \in (0,1)$, and suppose  $\R^m$ and $\R^n$ are endowed with the canonical Euclidean norms.  Then $H(A) = 1/\delta$.

On the other hand, for any fixed $b_2$ and $R:=\{x\in \R^2: \delta \cdot x_{2} \le b_2\}$ we have $H(A^1|R) = 1$.  That is, for all $b = \matr{b_1\\ b_2}$ and $u\in R$ we have
\[
\dist(u,P_A(b)) \le \dist(b_1-u,\R_+).
\]
\end{example}
}
}
\subsection{The case of inequalities and  equations}
\label{sec.equations.ineq}
Let $A \in \R^{m\times n}, \, C \in \R^{p \times n},$ and $R\subseteq \R^n$ be a reference polyhedron. Consider  systems of the form
\[
\begin{array}{l}
Ax \le b \\
Cx = d\\
x\in R,
\end{array}
\]
where $R$ represents some constraints that are easy to satisfy.

Proposition~\ref{prop.Hoffman.gral.2} below gives a bound analogous to~\eqref{eq:simple_erro} for the distance from a point $u\in R$ to a nonempty polyhedron of the form
\[
\{x\in R: Ax \le b, \, Cx = d\} = P_A(b) \cap C^{-1}(d) \cap R.
\]
For $J\subseteq[m]$ let  $[A,C; J]:\R^n \rightrightarrows \R^m\times \R^p$ be the set-valued mapping defined
by
\[
x \mapsto\{(Ax+s,Cx): s\in \R^m, \, s_J \ge 0 \}.
\]
Let
\[
\sJ(A,C\vert R):=\{(J,K): J \subseteq [m], \, K\in \T(R),\; [A,C;J](K) \text{ is a linear subspace}\},
\]
and
\begin{equation}\label{eq.Hoffman.const.rel}
H(A,C\vert R):=\dmax_{(J,K)\in \sJ(A,C\vert R)}
\dmax_{(y,w)\in \R^m\times C(K) \atop \|(y,w)\|\le 1} \dmin_{x\in K \atop A_Jx \le y_J, Cx = w} \|x\|.
\end{equation}

We have the following more general versions of Proposition~\ref{prop.Hoffman.rel},
Proposition~\ref{prop.Hoffman.rel.alt}, and Corollary~\ref{corol.Hoffman.rel}.

\begin{proposition}\label{prop.Hoffman.gral.2} Let  $R\subseteq \R^n$ be a reference polyhedron, $A \in \R^{m\times
n},$ and $C \in \R^{p \times n}$. Then for all $(b,d) \in \{(Ax+s,Cx): x\in R, s\in \R^m_+\}$ and $u\in R$
\begin{align}\label{eq.Hoffman.gral.2}
\dist(u,P_A(b)\cap C^{-1}(d)\cap R) &\le H(A,C\vert R)\cdot \dist\left((b-Au,d-Cu), \R^m_+\times\{0\}\right)
\end{align}
and this bound is tight: If $H(A,C\vert R) > 0$ then there exist $(b,d) \in \{(Ax+s,Cx): x\in R, s\in \R^m_+\}$ and $u
\in R$ such that
\[
\dist(u,P_A(b)\cap C^{-1}(d)\cap R) = H(A,C\vert R)\cdot \dist\left((b-Au,d-Cu),\R^m_+\times\{0\}\right)
>0.\]
\end{proposition}

\begin{proposition}\label{prop.Hoffman.gral.alt.2}
Let $R\subseteq \R^n$ be a reference polyhedron, $A \in \R^{m\times n},$ and $C \in \R^{p \times n}$. Then for all
$(J,K) \in \sJ(A,C\vert R)$
\[
\dmax_{(y,w)\in \R^m \times C(K) \atop \|(y,w)\|\le 1} \dmin_{x\in K \atop A_Jx \le y_J, Cx = w} \|x\| =
\frac{1}{\dmin_{v\in \R^J_+, z\in C(K) \atop \|(v,z)\|^* = 1, A_J\transp v + C\transp z-u\in K^*} \|u\|^*}.
\]
In particular
\[
H(A,C\vert R)= \dmax_{(J,K)\in \sJ(A,C\vert R)} \frac{1}{\dmin_{v\in \R^J_+, z\in C(K) \atop \|(v,z)\|^* = 1,
A_J\transp v +C\transp z-u\in K^*} \|u\|^*}.
\]
\end{proposition}

\begin{corollary}\label{corol.ineq.2}
 Let  $R\subseteq \R^n$ be a reference polyhedron, $A \in \R^{m\times n},$ and $C \in \R^{p \times n}$.

\begin{description}
\item[(a)] If $R$ is a cone and $\{(Ax+s,Cx): x\in R,\; s\in \R^m_+\}$ is a  linear subspace then
\begin{equation}\label{eq.HAC.rel}
H(A,C\vert R) =
\frac{1}{\dmin_{v\in \R^m_+, z\in C(R) \atop \|(v,z)\|^* = 1, A\transp v + C\transp z-u\in R^*} \|u\|^*}.
\end{equation}
\item[(b)]
  Suppose $\F\subseteq \sJ(A,C\vert R)$ and $\I \subseteq 2^{[m]}\times \T(R) \setminus \sJ(A,C\vert R)$
are such that for all $(J,K)\in 2^{[m]} \times \T(R)$ either  $J\subseteq F$ and $T \subseteq K$ for some
$(F,T) \in \F,$
or $I\subseteq J$ and $K\subseteq U$ for some $(I,U) \in \I.$   Then
\[
H(A,C \vert R) =  \dmax_{(J,K) \in \F}  \frac{1}{\dmin_{v\in \R^J_+, z\in C(K) \atop \|(v,z)\|^* = 1, A_J\transp v +
C\transp z-u\in K^*} \|u\|^*}.
\]
\end{description}
\end{corollary}

The constant $H(A,C\vert R)$ generalizes the previous constants $H(A)$, $H(A\vert R)$.
More precisely, by taking $p=0$ and $C = [\;]$ the ``empty'' $0\times n$ matrix, we get
\[
H(A,[\;]\,\vert\, R) = H(A\vert R).
\]
If $C = [\;]$ and $R = \R^n$ then we get $H(A,[\;]\,\vert\, \R^n) = H(A)$.

\medskip

Therefore Proposition~\ref{prop.Hoffman} and   Proposition~\ref{prop.Hoffman.rel} are special cases of Proposition~\ref{prop.Hoffman.gral.2}.  Likewise, 
Proposition~\ref{prop.Hoffman.alt} and  Proposition~\ref{prop.Hoffman.rel.alt} are special cases of Proposition~\ref{prop.Hoffman.gral.alt.2}.
We present the proofs of Proposition~\ref{prop.Hoffman.gral.2} and Proposition~\ref{prop.Hoffman.gral.alt.2}
in Section~\ref{sec.proof.props}.  They are immediate consequences of the more general Theorem~\ref{thm.main.gral} and
Theorem~\ref{thm.main.gral.alt} for polyhedral sublinear mappings.

\medskip

Another special and particularly interesting case occurs when $m = 0$ and $A=[\;]$.  This concerns systems of the
form
\[
\begin{array}{l}
Cx = d \\
x\in R,
\end{array}
\]
where $R$ represents some polyhedral constraints that are easy to satisfy.  In this case Proposition~\ref{prop.Hoffman.gral.2}
implies that for all $d\in C(R)$ and $x\in R$
\begin{equation}\label{eq.equations.rel}
\dist(x,C^{-1}(d)\cap R) \le 
\tilde H(C \, \vert \,R) \cdot \|d-Cx\|,
\end{equation}
where
\[
\tilde H(C\,\vert \,R) = H([\;],C \, \vert \,R) =  \dmax_{K\in \tilde\sJ(C \vert R)}
\dmax_{w\in C(K) \atop \|w\|\le 1} \dmin_{x\in K \atop Cx = w} \|x\|,
\]
and
\[
\tilde\sJ(C\, \vert \,R) =\sJ([\;],C\, \vert \,R) = \{K\in \T(R): C(K) \text{ is a linear subspace}\}.
\]
Proposition~\ref{prop.Hoffman.gral.2} also implies that the bound~\eqref{eq.equations.rel} is tight.  Furthermore,
Proposition~\ref{prop.Hoffman.gral.alt.2} yields
\[
\tilde H(C\vert R) = \dmax_{K\in \tilde\sJ(C \vert R)} \frac{1}{\dmin_{z\in C(K),\|z\|^* = 1 \atop  C\transp z-u\in K^*}\|u\|^*}.
\]
In addition, Corollary~\ref{corol.ineq.2} implies that if $R$ is a cone and $C(R)$ is a linear subspace then
\begin{equation}\label{eq.HA.eqns.equal}
\tilde H(C\, \vert R) = \max_{w\in C(R) \atop \|z\|\le 1} \min_{x\in R\atop Cx = w} \|x\| =
\frac{1}{\dmin_{z\in C(R),\|z\|^*=1\atop C\transp z -u \in R^*} \|u\|^*}.
\end{equation}

In a nice analogy to~\eqref{eq.HA} in Corollary~\ref{corol.sets}, the identity~\eqref{eq.HA.eqns.equal} has the following
geometric interpretation.  If $R$ is a cone then  $L := C(R)$ is a linear subspace if and only if
$Cx \in \ri(R)$ is feasible and when this is the case~\eqref{eq.HA.eqns.equal} implies that
\[
\frac{1}{\tilde H( C\, \vert \,R)} = \max\{r: w\in L, \|w\| \le r \Rightarrow y \in C(\B\cap R)\},
\]
where $\B:=\{x\in \R^n: \|x\|\le 1\}$.  In other words, $1/\tilde H(C \, \vert \,R)$ is the radius of the largest ball
in $L$ centered at the origin and contained in $L \cap C(\B\cap R)$.  This radius can be seen as a  generalization of the
smallest singular value of $C$. Indeed, observe that when $R= \R^n$ and both $\R^n$ and $\R^m$ are endowed with
Euclidean norms, $1/\tilde H(C \, \vert \,\R^n)$ is the smallest positive singular value of $C$.  In the special case
when $R$ is a cone and $C(R) = \R^m$, the quantity $1/\tilde H(C \, \vert \,R)$ equals the {\em distance to
non-surjectivity} of the mapping
\[
x \mapsto\left\{\begin{array}{ll}Cx & \text{ if } x\in R \\ \emptyset & \text{otherwise}\end{array}\right.
\]
as detailed in~\cite{Lewi99}.  This distance to non-surjectivity is the same as Renegar's {\em distance to
ill-posedness} of the system of constraints $Cx = 0, \, x\in \ri(R)$ defined by $C$ and $\ri(R)$.

\section{Computing Hoffman constants}
\label{sec.algo}

We next describe some algorithmic approaches to compute Hoffman constants.  For ease of exposition, we focus on the computation of $H(A)$ but the approaches described below
can be extended to compute or estimate more general relative Hoffman constants $H(A,C\vert R)$.

We describe two main approaches to compute $H(A)$.  The first approach is based on a formulation of $H(A)$ as a mathematical program with linear complementarity constraints (MPLCC).  The MPLCC formulation in turn can be rewritten as a mixed integer linear program or as a linear program with special order set constraints of type 1.  The second approach is based on identifying collections of sets that satisfy a certain {\em covering property} based on 
Corollary~\ref{corol.sets}.  

Throughout this section we assume that $A\ne 0$ as otherwise the computation of $H(A)$ is uninteresting.

\subsection{MPLCC formulation of $H(A)$}
\label{sec:numSolvers}
The next proposition shows that $H(A)$ can be formulated as an MPLCC.  

\begin{proposition}\label{prop:LCP} Let $A \in \R^{m\times n} \setminus\{0\}$. Then
\begin{align} \label{eq.LCP}
\begin{split}
 \frac 1{H(A)} = \min_{x,s,v}\; &\|A\transp v \|^*\\
\st\;& Ax - s \le -\1\\
& \|v\|^* = 1\\
& s_i v_i = 0,\; i=1,\dots,m\\
&  s \ge 0, \,v \ge 0.
\end{split}
 \end{align}
\end{proposition}
\begin{proof}
This is essentially a restatement of~\eqref{eq.Hoffman.alt}.  Indeed,~\eqref{eq.Hoffman.alt} can be rewritten as
\begin{equation}\label{eq.Hoffman.alt.rephrase}
 \frac 1{H(A)} = \min_{J \in \sJ(A)} \min_{u \in \R^J_+, \, \|u\|^* = 1} \|A_J\transp u \|^*.
 \end{equation}
To establish the equivalence between~\eqref{eq.LCP} and~\eqref{eq.Hoffman.alt.rephrase}, observe that the set of feasible points $(x,s,v)$ for~\eqref{eq.LCP} is in one-to-one correspondence with the set of pairs $(J,u)$ such that $J\in \sJ(A)$ and $u\in \R^J_+, \|u\|^* =1$ via 
\[
J = \{i\in[m]: s_i = 0\} \; \text{ and } \; u = v_J.
\]
Under this correspondence we have $\|A\transp v\|^* = \|A_J\transp u\|^*$ and thus~\eqref{eq.LCP}  is equivalent to~\eqref{eq.Hoffman.alt.rephrase}.
\end{proof}

For a suitable choice of norms, problem \eqref{eq.LCP} can be cast as a linear
program with linear complementarity constraints (LPLCC).  Indeed, suppose $\R^n$ and $\R^m$ are endowed with the $\ell_1$-norm and  $\ell_{\infty}$-norm respectively.  Then $\|A\transp v\|^* = \|A\transp v\|_\infty$ and $\|v\|^* = \|v\|_1$ for all $v\in \R^m$ and hence for this choice of norms
\eqref{eq.LCP} is equivalent to
\begin{align} \label{eq.LPCC}
\begin{split}
\frac 1{H(A)} = \dmin_{x,v,z, t} \; & t\\
\st \; & -t\1  \le A\transp v  \le t\1\\
& Ax - s \le -\1\\
& \1 \transp v = 1\\
& s_i v_i = 0,\; i=1,\dots,m\\
&  s \ge 0, \,v \ge 0
\end{split}
 \end{align}
LPLCC is a large and important class of problems that subsumes linear bilevel optimization and non-convex quadratic programming among others.  There a variety of solution methods for LPLCCs, many of them based on enumerative schemes.  For a detailed review on this subject, see~\cite{Judi12}.  We next describe how~\eqref{eq.LPCC} can also be formulated as a mixed integer linear program and as a linear program with special order set constraints of type 1.

By using big-M constraints, we can  reformulate~\eqref{eq.LPCC} as
the following mixed integer linear program (MILP):
\begin{align}
\label{eq:MILPprob}
\begin{split}
\dmin_{x,v,z, t} \; & t\\
\st \; & A\transp v \le \1t\\
& A\transp v \ge -\1t\\
&Ax \le -z + M(\1-z) \\
& \1\transp v = 1 \\
& 0\le v \le z \\
&z_j \in \{0,1\}, \; \; j=1,\dots,m.
\end{split}
\end{align}
A potential limitation of~\eqref{eq:MILPprob} is the need for an appropriate and valid estimate for the value of $M$.  Modern MILP solvers provide some ways to overcome this limitation. First, state-of-the-art MILP solvers enable the alternate reformulation of the third and sixth constraints in~\eqref{eq:MILPprob} as the following set of {\em indicator} constraints (see~\cite[Chapter 26]{studio2013users}):
\begin{equation}
\label{eq:INDtrick}
\begin{array}{ll}
z_j = 1 \Rightarrow A_{j} x \le -1, & j=1,\dots,m.\\
z_j \in \{0,1\}, & j=1,\dots,m.
\end{array}
\end{equation} 

Another alternative to big-M constraints is to use {\em special order set constraints of type 1} (SOS1) as
discussed in~\cite{pineda2018efficiently, siddiqui2013sos1}.  An SOS1 constraint is a set of variables in which at most one member can be strictly positive.
Problem~\eqref{eq.LPCC} can be reformulated as a the following linear program with SOS1
constraints:
\begin{equation}
\label{eq:MILPsos}
\begin{array}{lllll}
\dmin_{x,v,t,s} \; & t\\
\operatorname{s.t.} \; & -t\1 \le A\transp v \le t\1\\
& Ax - s \le -\1 \\
& \1\transp v = 1 \\
& s \ge 0,  \; v\ge 0 \\
& \{v_j,s_j\} \in \text{SOS1}, & j=1,\dots,m.
\end{array}
\end{equation}

\subsection{Computation of $H(A)$ via 
the {\em covering property}}\label{sec:certificates}

Let $A\in \R^{m\times n}\setminus\{0\}$.  Corollary~\ref{corol.sets} suggests the following algorithmic approach to
compute $H(A)$.  Find $\F \subseteq \sJ(A)$ and $\I \subseteq 2^{[m]} \setminus \sJ(A)$ that satisfy the
following {\em covering property:}

\begin{quote}
For all $J\in 2^{[m]}$ either $J\subseteq F$ for some $F\in \F$ or $I\subseteq J$ for some $I\in \I$.
\end{quote}
Then compute
\begin{equation}\label{eq.Hoffman.small}
H(A):= \max_{J\in\F} \, \frac{1}{\min\{\|A_J\transp v\|^*: v\in \R^J_+, \|v\|^* = 1\}}.
\end{equation}
We choose the term {\em covering property} since the above condition can be alternatively stated as follows: every element $J\in 2^{[m]}$ of the the  ground set $2^{[m]}$ is either ``covered'' by some set in  $\F$ or its complement $[m]\setminus J$ is ``covered'' by the complement of some set in $\I$.  Observe that $(\F,\I) = (\overline \sJ(A),\underline \sJ(A))$ 
satisfies the covering property for the collection
$\overline \sJ(A) \subseteq \sJ(A)$ of maximal (inclusion-wise) sets in $\sJ(A)$ and the collection $\underline \sJ(A)\subseteq 2^{[m]}\setminus \sJ(A)$  of minimal (inclusion-wise) sets in $2^{[m]}\setminus\sJ(A)$.  Furthermore, it is easy to see that if  $(\F,\I)$ satisfies the covering property then $\overline \sJ(A) \subseteq \F$ and $\underline \sJ(A) \subseteq \I$.  In other words, $(\overline \sJ(A),\underline \sJ(A))$ is the minimal pair of collections that satisfies the covering property.


The main challenge in computing $H(A)$ via~\eqref{eq.Hoffman.small} is the identification of suitable collections $\F$ and
$\I$ that satisfy the  covering property.
In some special cases, it is possible to find $\overline \sJ(A)$ analytically and thus compute $H(A)$ via~\eqref{eq.Hoffman.small} with $\F = \overline \sJ(A)$. We illustrate this approach via some examples in
Section~\ref{sec:analytical}. In Section~\ref{sec:procedure} Algorithm~\ref{alg:bb} describes a general procedure  that gradually constructs $\overline\sJ(A)$ and $\underline \sJ(A)$.  We should note that Algorithm~\ref{alg:bb} is only viable when the collections 
$\overline\sJ(A)$ and $\underline \sJ(A)$ are of reasonable size since the algorithm constructs both of these collections explicitly.

The following observation facilitates the computation of $H(A)$.  If $\R^n$ and~$\R^m$ are  endowed with the
$\ell_1$-norm and $\ell_\infty$-norm respectively then
\begin{equation}\label{eq.norm.lp}
\min\{\|A_J\transp v\|^*: v\in \R^J_+, \|v\|^* = 1\}=
\min\{\|A_J\transp v\|_{\infty}: v\in \R^J_+, \1\transp v = 1\}.
\end{equation}
The latter expression is computable via linear programming.  

It is worthwhile noting that although the computation of $H(A)$ depends on the choices of norms, the covering property does not.  In particular, 
if the pair $(\F,\I)$ satisfies the covering property, then $H(A)$ can be computed or estimated for any choice of norms via~\eqref{eq.Hoffman.small} provided that $\min\{\|A_J\transp v\|^*: v\in \R^J_+, \|v\|^* = 1\}$ can be computed or estimated.  We also note that if the pair $(\F,\I)$ satisfies the covering property then it provides a {\em certificate of optimality} for $H(A)$ since it certifies that $H_J(A) \le H(A)$ for any $J\in \sJ(A)$.  By contrast, the MPLCC and MILP approaches described in Section~\ref{sec:numSolvers} do not readily provide such a certificate of optimality for $H(A)$.

\subsubsection{Some examples}\label{sec:analytical}

As the following examples illustrate, when the matrix $A$ is highly structured it may be  possible to
construct $\F$ and $\I$ directly so that the covering property holds.  
In the examples below we actually identify the collection $\overline \sJ(A)$ of maximal sets in $\sJ(A)$.
Once $\overline \sJ(A)$ is identified, we can compute $H(A)$ via~\eqref{eq.Hoffman.small} with $\F = \overline \sJ(A)$.  
To facilitate the latter computation, in the next three examples we assume that $\R^n$ and~$\R^m$ are endowed with the $\ell_1$-norm and $\ell_\infty$-norm respectively so that 
\eqref{eq.norm.lp} holds.

\begin{example}[box] 
\label{ex.box}
Let $n \ge 1$ and consider the matrix
\[
A = \matr{\;\;\Id_n \\ -\Id_n} \in \R^{2n\times n}
\]
where $\Id_n\in\R^{n\times n}$ denotes the $n\times n$ identity matrix.  In this case the collection $\overline \sJ(A)$  consists of the $2^n$ sets of the form
$\{i_1,i_2,\dots,i_n\}$ where $i_k\in \{k,k+n\}, \; k=1,\dots,n$.  A straightforward calculation then shows that
\[
H(A) = \max_{J \in \overline \sJ(A) }  \frac{1}{\dmin_{v\in \R^J_+, \1\transp v=1}\|A_J\transp v\|_\infty} =
 \frac{1}{\dmin_{v\in \R^n_+, \1\transp v=1}\|v\|_\infty} = n.
\]
\end{example}

\begin{example}[simplex]\label{ex.simplex}
Let $n \ge 1$ and consider the matrix
\[
A = \matr{\;\;\Id_n \\ -\1\transp} \in \R^{(n+1)\times n}.
\]
In this case the collection
$\overline\sJ(A)$  consists of the $n+1$ sets $\{1,\dots,n\}$ and $\{1,\dots,n,n+1\}\setminus\{k\}, \; k=1,\dots,n$.  A straightforward calculation then shows that
\begin{align*}
H(A) &= \max\left(
 \frac{1}{\dmin_{v\in \R^n_+, \1\transp v=1}\|v\|_\infty}, \frac{1}{\dmin_{(v,v_n)\in \R^n_+,  \1\transp v + v_n=1}\|(v-v_n\1,-v_n)\|_\infty} \right) \\
 &= \max(n,2n+1)\\
 &=2n+1.
\end{align*}

\end{example}

\begin{example}[$\ell_1$-unit ball]
\label{ex.cube} Let $n \ge 1$ and consider the matrix
$
A \in \R^{2^n \times n}
$
whose rows are the $2^n$ vectors in $\{-1,1\}^n$ ordered lexicographically, that is,
\[
A = \matr{-1 & -1 & \cdots &-1 & -1\\
-1 & -1 & \cdots &-1 & \;\,\;1\\
-1 & -1 & \cdots &\;\,\;1 & -1\\
\;\,\;\vdots & \;\,\;\vdots & \ddots &\;\,\;\vdots&\;\,\;\vdots\\
\;\,\;1 & \;\,\;1 & \cdots &\;\,\;1&-1 \\
\;\,\;1 & \;\,\;1 & \cdots &\;\,\;1&\;\,\;1}.
\]
In this case the collection $\overline\sJ(A)$  consists of the collection of all sets of the form
\[
J_u = \{i: A_iu < 0\}
\]
for $u\in \R^n$ such that all components of $Au$ are non-zero.  
The symmetry of $A$ implies that
\begin{equation}\label{eq.hoffman.red}
H(A) = \max_{J \in \overline\sJ(A) }  \frac{1}{\dmin_{v\in \R^J_+, \1\transp v=1}\|A_J\transp v\|_\infty} = 
\max_{J \in \sJ_0 }  \frac{1}{\dmin_{v\in \R^J_+, \1\transp v=1}\|A_J\transp v\|_\infty},
\end{equation}
where $\sJ_0\subseteq \F$ is the smaller collection of sets of the form $J_u$ where $u\in \R^n$ has non-increasing entries and all entries of $Au$ are non-zero.  For small values of $n$, both $\sJ_0$ and   $H(A)$  can be computed explicitly.  The following table displays the values of $H(A)$ and the optimal index set $J\in \sJ_0$ where~\eqref{eq.hoffman.red} attains its maximum for $n=1,2,3,4,5$.  To ease notation, the $J$ entry in each column only displays the new indices that need to be added to the $J$ entry in the previous column which is denoted by `$\cdots$': 
\[
\begin{array}{c||c|c|c|c|c} n & 1 & 2 & 3 & 4 & 5 \\
\hline \hline
H(A) & 1 & 1 & 3 & 5 & 9 \\
\hline
J & \{1\} & \cdots \cup \{2\} & \cdots \cup\{3,5\} & \cdots \cup\{4,6,7,9\} &   \cdots \cup 
\{ 8,10,11,12,13,   17, 18,    19\}
\end{array}
\]
\end{example}
The values of the Hoffman constant $H(A)$ in Example~\ref{ex.box}, Example~\ref{ex.simplex}, and Example~\ref{ex.cube} can also be obtained via   the MPLCC formulations discussed in Section~\ref{sec:numSolvers} above or via Algorithm~\ref{alg:bb} below.

\subsubsection{An algorithm that constructs $\F$ and $\I$ gradually}\label{sec:procedure}
Algorithm~\ref{alg:bb} below formalizes the following simple iterative procedure to construct a pair $(\F,\I)$ that satisfies the covering property: Start with $\F = \I = \emptyset$.  At each subsequent iteration check whether $(\F,\I)$ covers $2^{[m]}$.  It is does, then we are done.  Otherwise, find $J \in 2^{[m]}$ that is not covered by $(\F,\I)$ and check whether $J\in \sJ(A)$.  If indeed $J\in \sJ(A)$ then add $J$ to $\F$.  Otherwise, find $I \subseteq J$ such that $I\not \in \sJ(A)$ and add $I$ to $\I$.  This procedure must eventually terminate since each iteration 
adds a new element to $\F$ or to $\I$.  Furthermore, as Proposition~\ref{prop.algo} below shows, if this procedure is properly executed, it constructs the minimal collections $\F = \overline \sJ(A)$ and $\I = \underline \sJ(A)$ that satisfy the covering property.

The algorithm needs to perform two main steps.  First, given 
$\F\subseteq \sJ(A)$ and $\I \subseteq 2^{[m]}\setminus \sJ(A)$, find $J \in 2^{[m]}$ not covered by $(\F,\I)$ or verify that no such $J$ exists.   
Second, given $J \in 2^{[m]}$, either certify that $J\in\sJ(A)$ or else find $I\subseteq J$ with $I \in 2^{[m]}\setminus\sJ(A)$.

The first step can be accomplished by solving the following combinatorial optimization problem
\begin{equation}\label{eq.ip}
\begin{array}{rl}
\dmax_J & |J|\\
\st & |J^c\cap I| \ge 1, \; I\in \I\\
& |J\cap F^c| \ge 1, \; F\in \F\\
& J \subseteq [m].
\end{array}
\end{equation}
Observe that $(\F,\I)$ satisfy the covering property if and only if~\eqref{eq.ip} is infeasible.  Otherwise the optimal solution of~\eqref{eq.ip} yields $J\in 2^{[m]}$ of maximal size that is not covered by $(\F,\I)$.

The second step can be accomplished by solving the following optimization problem
\begin{equation}\label{eq.lp}
\min\{\|A_J\transp v\|^*: v\in \R^J_+, \|v\|^* = 1\}.
\end{equation}
Observe that $J\in\sJ(A)$ if and only if the optimal value of~\eqref{eq.lp} is positive.  Otherwise, when the optimal
value of~\eqref{eq.lp} is zero, its optimal solution satisfies $v\in \R^J_+\setminus\{0\}$ and $A_J\transp v = 0$.  In
this case $I(v):=\{i\in J: v_i > 0\} \in 2^{[m]}\setminus \sJ(A)$ and $I(v)\subseteq J$.  Furthermore, for additional efficiency we will assume that in the latter case $v$ is chosen so that $I(v)$ is of minimal size.  This condition is easily enforceable by applying a straightforward post-processing procedure whenever the optimal value of~\eqref{eq.lp} is zero.
As we noted in~\eqref{eq.norm.lp}, when $\R^n$ and~$\R^m$ are  endowed with the
$\ell_1$-norm and $\ell_\infty$-norm respectively problem~\eqref{eq.lp} can be rewritten as a linear program.

%
\begin{algorithm}
  \caption{
   Computation of a pair $(\F ,\I)$ satisfying the covering property and $H(A)$}
\label{alg:bb}
\begin{algorithmic}[1]
\State {\bf input} $A \in \R^{m \times n}\setminus\{0\}$
\State Let $\F := \emptyset, \;\I := \emptyset, H(A):= 0$
\While {$(\F,\I)$ does not satisfy the covering property}
	\State Solve~\eqref{eq.ip} to pick $J \in 2^{[m]}$ not covered by $(\F,\I)$
	\State Let $v$ solve~\eqref{eq.lp} to detect whether $J\in \sJ(A)$
		\If {$\|A_J\transp v\|^* > 0$}
			\State
			$\F := \F \cup \{J\}$  and  $H(A) :=
\max\left\{H(A),\frac{1}{\|A_J\transp v\|^*}\right\}$
		\Else
			\State
			 Let $\I:= \I \cup\{I(v)\}$		
	    \EndIf
		\EndWhile
\State \Return $\F, \, \I, \, H(A)$
\end{algorithmic}
\end{algorithm}

It is easy to see that if $[m]\in \sJ(A)$ then Algorithm~\ref{alg:bb} terminates after one iteration: the
first iteration places $[m]$ in $\F$ and leaves $\I$ empty.  At that point the collections $\F =
\overline \sJ(A) = \{[m]\}$ and  $\I = 
\underline \sJ(A) =
\emptyset$ satisfy the covering property.   Similarly when $[m]$ is the only
set missing from $\sJ(A)$ Algorithm~\ref{alg:bb}  terminates after $1+m$ iterations:  the first iteration places the
set $[m]$  in $\I$ and the next $m$ iterations place the sets $[m]\setminus \{j\}$ for
$j=1,\dots,m$ in $\F$.
At that point the collections $\F  = \overline \sJ(A) = \{[m]\setminus\{j\}: j=1,\dots,m\}$ and $\I = \underline \sJ(A) = \{[m]\}$ satisfy
the covering property.


For general $A\in\R^{m\times n} \setminus\{0\}$, Proposition~\ref{prop.algo} extends the above two special cases.  We should note that, as the above two cases illustrate, the expression $|\overline \sJ(A)| +  \vert\underline \sJ(A)\vert =  \vert\overline \sJ(A)\cup\underline \sJ(A)\vert$ in~\eqref{eq.algo1.bound} below is always smaller, and in some cases vastly smaller, than $2^m$ which is the number of steps that a brute force scan of the subsets of $[m]$ would require. Indeed, observe that for each $J\in \overline \sJ(A) \cup \underline \sJ(A)$ one of the following two situations must occur: {\em either} $J \in 
\overline \sJ(A)$ and hence $\overline \sJ(A) \cup \underline \sJ(A)$ has none of the proper subsets of $J$, {\em or} $J \in 
\underline \sJ(A)$ and hence $\overline \sJ(A) \cup \underline \sJ(A)$ has none of the proper supersets of $J$.
\begin{proposition}\label{prop.algo} Let $A\in \R^{m\times n} \setminus\{0\}$ and let $\overline \sJ(A) \subseteq \sJ(A)$ denote the collection of maximal (inclusion-wise) sets in $\sJ(A)$ and let $\underline \sJ(A) \subseteq 2^{[m]}\setminus\sJ(A)$ denote the collection of minimal (inclusion-wise) sets in  $2^{[m]}\setminus\sJ(A)$.  Then Algorithm~\ref{alg:bb}
terminates after 
\begin{equation}\label{eq.algo1.bound}
|\overline \sJ(A)| +  \vert\underline \sJ(A)\vert
\end{equation} iterations.  Furthermore, upon termination  Algorithm~\ref{alg:bb}  returns $\F = \overline \sJ(A)$ and $\I = \underline \sJ(A)$.
\end{proposition}
\begin{proof}
Each iteration of Algorithm~\ref{alg:bb} either adds one element $J\in \overline \sJ(A)$ to $\F$ because the set $J$ selected in Step 4 is of maximal size, or adds one element from $I(v) \in \underline \sJ(A)$ to $\I$ because the set $I(v)$ selected in Step 9 is of minimal size.  
To finish, observe that Algorithm~\ref{alg:bb}
terminates as soon as $\F = \overline \sJ(A)$ and $\I = \underline \sJ(A)$.

\end{proof}

The expression~\eqref{eq.algo1.bound} can be interpreted as follows.  Suppose $\overline \sJ(A) = \{J_1,\dots,J_k\}$.  Then $J\in 2^{[m]}\setminus \sJ(A)$ if and only if $J\setminus J_i \ne \emptyset$ for $i=1,\dots,k$.  Therefore $\vert \underline \sJ(A)\vert$ is the number of minimal (componentwise) solutions $(x_1,\dots,x_m) \in \{0,1\}^m$ to the following system of set-covering constraints
\[
\begin{array}{rl}
\dsum_{i\not\in J_\ell} x_i \ge 1 & \text{for} \; \ell=1,\dots,k \\
x_i \in \{0,1\}& \text{for} \; i=1,\dots,m.
\end{array}
\]
Hence  $|\overline \sJ(A)|+\vert \underline \sJ(A)\vert$ will not be too large if the sets in $\overline \sJ(A)$ are few and large. 
 The most favorable case $|\overline \sJ(A)|+\vert \underline \sJ(A)\vert=1$  occurs when $\overline \sJ(A) = \{[m]\}$ or equivalently when $Ax < 0$ is feasible.  The next most favorable case $|\overline \sJ(A)|+\vert \underline \sJ(A)\vert=1+m$ occurs when $\underline \sJ(A) = \{[m]\}$ or equivalently when 
$Ax < 0$ is infeasible but it becomes feasible as soon as we drop one inequality.

\subsection{Numerical experiments}
Table~\ref{tab:real} and Table~\ref{tab:exs} summarize two sets of numerical experiments that illustrate and compare three different methods to compute $H(A)$: the SOS1 formulation~\eqref{eq:MILPsos}, Algorithm~\ref{alg:bb} based on the covering property, and an enumeration scheme based on~\eqref{eq.popular} via a scan of the subsets  $J\subseteq [m]$ such that $A_J$ is full row rank. We will refer to these three methods as SOS1, COVER, and ENUM respectively. To make the implementation of ENUM more efficient, we used a variant of Algorithm~\ref{alg:bb} that scans only for maximal subsets  $J\subseteq [m]$ such that $A_J$ has full row rank.  In particular, the number of sets that ENUM needs to scan is  bounded above by 
$
{m \choose n} +  {m \choose n-1} + \cdots + {m \choose 1}
$. All experiments were carried out in an iMac with a 3.5GHz Intel Core i7 and 32 GB 1600 MHz DDR3 RAM, using {\tt MATLAB R2017b} and {\tt Gurobi 8.1.1}.

Table~\ref{tab:real}  reports results on a set of instances $A\in \R^{m\times n}$ with $n \le m \le 100$ drawn from {\tt LPnetlib, CUTEr,} and {\tt Globallib} test sets.    
For each instance we display the value $H(A)$ and the  following number of ``main steps'' performed by each method.  For SOS1, we display the number of nodes (including the root node) that {\tt Gurobi} generated to solve~\eqref{eq:MILPsos}.  For COVER, we display the number of main iterations.  For ENUM we  display the number of sets $J\subseteq [m]$ that had to be scanned.  The main task required at each main step in each method is the solution of a linear program.  Thus the number of main steps gives a rough comparison among the three methods.  An entry `$\ast$' in Table~\ref{tab:real} indicates that the method did not solve the corresponding instances within 10,000 main steps.

Although we ran over 150 instances, we only display results for the more challenging instances where either SOS1 or COVER require more than one main step.  The entries in Table~\ref{tab:real} (which is ordered by the number of steps taken by the SOS1 method) reveal that many of the instances from the {\tt LPnetlib, CUTEr,} and {\tt Globallib} are easily solved by SOS1 and COVER whereas only the very smallest ones are solved by ENUM. 
The easiest instances (about 120 total, not resported in the Table) are those with $\underline \sJ(A) = \{[m]\}$ which COVER solves in one step and SOS1 generally solves in one or two steps.  By contrast, ENUM cannot solve most of these instances within 10,000 steps.  The first, fifth, and sixth instances in Table~\ref{tab:real} illustrate this behavior.  For other more challenging instances, COVER usually finds $H(A)$ after fewer main steps than SOS1. However, the SOS1 ``steps'' (nodes) are much faster (up to an order of magnitude) than the COVER iterations.  This is not surprising since we used {\tt Gurobi}, a state-of-the-art mathematical programming solver, in our SOS1 implementation.  By contrast, we used fairly straightforward and simple {\tt MATLAB} implementations for both COVER and ENUM.

\begin{table}[!htb]
\begin{center}
\begin{tabular}{lrrcrcrcrrrcrrr}
\toprule
Instance & $m$ & $n$ && \multicolumn{1}{r}{$H(A)$} && \multicolumn{1}{c}{SOS1} &&  \multicolumn{1}{c}{COVER}  &&  \multicolumn{1}{c}{ENUM} \\
\midrule
qp50\_25\_3\_3.mat	&	50	&	25	&&	2.7203	&&	2	&&	1	&&		\multicolumn{1}{r}{$\ast$}	\\
st\_bpv2.mat	&	5	&	4	&&	2.0000	&&	2	&&		3	&&		3	\\
st\_iqpbk1.mat	&	8	&	7	&&	28.0000	&&	2	&&	9	&&		8	\\
st\_fp5 	&	11	&	10	&&	3.2286	&&	6	&&		3	&&		3	\\
qp50\_25\_1\_3 	&	50	&	25	&&	3.5817	&&	8	&&		1	&&	\multicolumn{1}{r}{$\ast$}	\\
qp40\_20\_4\_2 	&	40	&	20	&&	10.4870	&&	10	&&		1	&&	\multicolumn{1}{r}{$\ast$}	\\
st\_e22 	&	5	&	2	&&	11.0000	&&	26	&&		10	&&	10	\\
st\_glmp\_kk92 	&	6	&	4	&&	4.0000	&&	30	&&	11	&&	11	\\
st\_glmp\_fp2 	&	7	&	4	&&	134.3334	&&	38	&& 18	&&		22	\\
st\_e25 	&	8	&	4	&&	173.2708	&&	39	&&		13	&&		70	\\
qp20\_10\_4\_2 	&	20	&	10	&&	36.2748	&&	40	&&		6	&&		\multicolumn{1}{r}{$\ast$}\\
qp30\_15\_1\_1 	&	30	&	15	&&	3411.7169	&&	40	&&	8	&&		\multicolumn{1}{r}{$\ast$}\\
st\_glmp\_ss1 	&	8	&	5	&&	9.0000	&&	58	&&		20	&&		20	\\
st\_glmp\_kky 	&	8	&	7	&&	3.1429	&&	64	&&	20	&&	20	\\
qp20\_10\_3\_1 	&	20	&	10	&&	49.7414	&&	88	&&	15	&&	\multicolumn{1}{r}{$\ast$}\\
st\_ph20 	&	9	&	3	&&	28.0000	&&	108	&&	33	&&	84	\\
st\_ph13 	&	10	&	3	&&	6.0000	&&	118	&&	15	&&	99	\\
st\_qpk2 	&	12	&	6	&&	2.0000	&&	248	&&	70	&& 70	\\
biggsc4 	&	13	&	4	&&	5.0000	&&	420	&&		61	&&		385	\\
qp1 	&	50	&	1	&&	483.3333	&&	1026	&&	501	&&		50\\
qp3 	&	100	&	1	&&	483.3333	&&	1026	&&		551	&&		100\\
st\_qpk3 	&	22	&	11	&&	1.8196	&&	7046	&&	2059 &&	2059\\
lpi\_woodinfe 	&	89	&	35	&&	57.0006	&&	8005	&&		\multicolumn{1}{r}{$\ast$}&&	\multicolumn{1}{r}{$\ast$}\\
qp30\_15\_3\_1 	&	30	&	15	&& 104528.1035	&& 	\multicolumn{1}{r}{$\ast$}	&&		371	&&	\multicolumn{1}{r}{$\ast$}&\\
\bottomrule
\multicolumn{12}{l}{\scriptsize $\ast$: Algorithm reached limit on the number of steps.}\\
\end{tabular}
\end{center}
\caption{Performance of SOS1, COVER, and ENUM on test instances from {\tt LPnetlib, CUTEr, Globallib}, with a limit of 10,000 main steps.}
\label{tab:real}
\end{table}

Table~2 reports results on instances of $A$ generated as in Example~\ref{ex.box} and Example~\ref{ex.cube}.  For each of these two examples we included two instances.  The first one corresponds to the largest $n$ such that the instance was solvable by all three methods within 10,000 steps. The second one 
corresponds to the very next value of $n$.  The general conclusions of these four experiments are that ENUM appears to be the least effective method, while the relationship between the efficiency of SOS1 and COVER is not always the same.  It is noteworthy that Example~\ref{ex.cube} poses a significant challenge even for SOS1 when $n=5$ thereby highlighting the formidable challenge that computing $H(A)$ entails.  We also tested the methods for instances of $A$ generated as in Example~\ref{ex.simplex} for $n$ up to 1000.  Not surprisingly, all three methods solve those instances quite easily: SOS1 usually generates only one node beyond the root node, COVER terminates after $n+2$ iterations, and ENUM terminates after scanning $n+1$ sets.

\begin{table}[!htb]
\begin{center}
\begin{tabular}{lrrcrcrcrrrcrrr}
\toprule
Instance & $m$ & $n$ && \multicolumn{1}{c}{$H(A)$} && \multicolumn{1}{c}{SOS1} &&  \multicolumn{1}{c}{COVER}  &&  \multicolumn{1}{c}{ENUM} \\
\midrule
Example 1	&	26	&	13	&&	13.0000	&&	24	&&		8205		&&		8205& \\
Example 1	&	28	&	14	&&	14.0000	&&	24	&&		$*$		&&	$*$\\
Example 3	&	16	&	4	&&	5.0000	&&	1552	&&		152		&&		1128 \\
Example 3	&	32	&	5	&&	9.0000	&&	\multicolumn{1}{r}{$\ast$}	&&		4594		&&		$*$ \\\bottomrule
\end{tabular}
\end{center}
\caption{Performance of SOS1, COVER, and ENUM on instances of $A$ generated as in Example~\ref{ex.box} and Example~\ref{ex.cube}.}
\label{tab:exs}
\end{table}

We  reiterate that the above results are based on entirely straightforward and simple implementations of SOS1, COVER, and ENUM.  In particular, our basic implementations do not use warm-starts and do not make any attempt to exploit the symmetry structure that is evident in Example~\ref{ex.box}, Example~\ref{ex.simplex}, and Example~\ref{ex.cube}.  The development of more sophisticated implementations of the above algorithmic schemes will be an interesting topic for future work.

\section{A Hoffman constant for polyhedral set-valued mappings}
\label{sec.proof}

We next present a characterization of the Hoffman constant for polyhedral set-valued mappings.
Recall that a set-valued mapping $\Phi: \R^n \rightrightarrows \R^m$ assigns a set $\Phi(x) \subseteq \R^m$ to each
$x\in \R^n.$
Let $\Phi:\R^n \rightrightarrows\R^m$ be a set-valued mapping. The inverse $\Phi^{-1}:\R^m \rightrightarrows \R^n$ of
$\Phi$ is the set-valued mapping defined in the following natural way
\[
x\in \Phi^{-1}(y) \text{ if and only if } y\in \Phi(x).
\]
We will say that $\Phi$ is {\em polyhedral}  if \[\graph(\Phi) := \{(x,y): y\in \Phi(x)\} \subseteq \R^n \times \R^m\]
is a polyhedron. We will say that $\Phi$ is {\em sublinear} if $\graph(\Phi)$ is a convex cone, and we will say that
$\Phi$ is {\em closed} if $\graph(\Phi)$ is closed.  In particular, we will say that $\Phi$ is a {\em polyhedral
sublinear mapping} if  $\graph(\Phi)$ is a polyhedral convex cone.

Let $\Phi:\R^n \rightrightarrows\R^m$ be a set-valued mapping.  The domain and image of $\Phi$ are defined as follows:
\begin{align*}
\dom(\Phi) &= \{x\in \R^n: (x,y) \in \graph(\Phi) \text{ for some } y\in \R^m\},\\
\Image(\Phi) &= \{y\in \R^m: (x,y) \in \graph(\Phi) \text{ for some } x\in \R^n\}.
\end{align*}
When $\Phi:\R^n\rightrightarrows \R^m$ is a sublinear mapping, the norm of
$\Phi$ is defined as follows
\[
\|\Phi\| = \dmax_{x\in \dom(\Phi) \atop \|x\|\le 1} \min_{y\in \Phi(x)} \|y\|.
\]
In particular, if $\Phi:\R^n\rightrightarrows \R^m$ is a sublinear mapping then the norm of
$\Phi^{-1}:\R^m\rightrightarrows \R^n$ is
\[
\|\Phi^{-1}\| = \dmax_{y\in \dom(\Phi^{-1})\atop \|y\|\le 1} \min_{x\in \Phi^{-1}(y)} \|x\|= \dmax_{y\in
\Image(\Phi)\atop \|y\|\le 1} \min_{x\in \Phi^{-1}(y)} \|x\|.
\]
It is easy to see that both $\|\Phi\|$ and $\|\Phi^{-1}\|$ are finite if  $\Phi:\R^n \rightrightarrows\R^m$ is a
polyhedral sublinear mapping.

Suppose $\Phi:\R^n\rightrightarrows \R^m$ is sublinear.  Then the {\em upper adjoint} $\Phi^*:\R^m\rightrightarrows
\R^n$ is defined as follows
\[
u\in \Phi^*(v) \Leftrightarrow \ip{u}{x}\le \ip{v}{y} \text{ for all } (x,y) \in \graph(\Phi) \Leftrightarrow (-u,v) \in \graph(\Phi)^*.
\]
We will rely on the following correspondence between polyhedral sublinear mappings and polyhedral cones.  By
definition, $\Phi:\R^n \rightrightarrows \R^m$ is a polyhedral sublinear mapping if $\graph(\Phi) \subseteq \R^n\times
\R^m$ is a polyhedral cone.  Conversely, if $K\subseteq \R^n\times \R^m$ is a polyhedral cone, then the set-valued
mapping $\Phi_K:\R^n \rightrightarrows \R^m$ defined via
\[
y\in \Phi_K(x) \Leftrightarrow (x,y) \in K
\]
is a polyhedral sublinear mapping with $\graph(\Phi_K) = K$.

A polyhedral set-valued mapping $\Phi: \R^n \rightrightarrows \R^m$ is {\em surjective} if $\Image(\Phi) = \R^m.$
More generally, we will say that $\Phi$ is {\em relatively surjective} if $\Image(\Phi)$ is a linear subspace.
Suppose $\Phi:\R^n \rightrightarrows \R^m$ is a polyhedral set-valued mapping. Let $$\S(\graph(\Phi)):=\{T\in
\T(\graph(\Phi)): \Phi_T \text{ is relatively surjective}\},$$ and $$\H(\Phi):=\max_{T\in \S(\graph(\Phi))}
\|\Phi_T^{-1}\|.$$
We have the following general versions of Proposition~\ref{prop.Hoffman}, Proposition~\ref{prop.Hoffman.alt}, and
Corollary~\ref{corol.sets} for polyhedral set-valued mappings.

\begin{theorem}\label{thm.main.gral} Let $\Phi:\R^n \rightrightarrows \R^m$ be a polyhedral set-valued mapping.  Then
for all $b\in \Image(\Phi)$ and $u\in \dom(\Phi)$
\begin{equation}\label{eq.Hoffman.set}
\dist(u,\Phi^{-1}(b))\le \H(\Phi)\cdot \dist(b,\Phi(u)).
\end{equation}
Furthermore, this bound is tight: If $\H(\Phi)>0$ then there exist $b\in \Image(\Phi)$ and $u\in \dom(\Phi)$ such that
\[
\dist(u,\Phi^{-1}(b)) = \H(\Phi)\cdot \dist(b,\Phi(u))>0.
\]
\end{theorem}

\begin{theorem}\label{thm.main.gral.alt}
Let $\Phi:\R^n\rightrightarrows\R^m$ be a  polyhedral set-valued mapping.  Then for all $T\in \S(\graph(\Phi))$
\[
\|\Phi_T^{-1}\| = \dmax_{u\in \Phi_T^*(v)\atop \|u\|^*\le 1} \|\Pi_{\Image(\Phi_T)}(v)\|^* =
\frac{1}
{\dmin_{u\in \Phi_T^*(v)\atop\|\Pi_{\Image(\Phi_T)}(v)\|^*=1}\|u\|^*},
\]
where $\Pi_{\Image(\Phi_T)}:\R^m\rightarrow \Image(\Phi_T)$ denotes the orthogonal projection on the subspace
${\Image(\Phi_T)}$.
In particular,
\[
\H(\Phi) = \max_{T\in \S(\graph(\Phi))}\frac{1}
{\dmin_{u\in \Phi_T^*(v)\atop\|\Pi_{\Image(\Phi_T)}(v)\|^*=1}\|u\|^*}.
\]
\end{theorem}

\begin{corollary}\label{corol.phi.sets}
Let $\Phi: \R^n \rightrightarrows \R^m$ be a polyhedral set-valued mapping.

\begin{description}
\item[(a)] If $\Phi$ is sublinear and relatively surjective then
\[\H(\Phi) = \frac{1}{\dmin_{u\in \Phi^*(v)\atop\|\Pi_{\Image(\Phi)}(v)\|^*=1}\|u\|^*}.
\]
\item[(b)]
  Suppose $\mathfrak{F}\subseteq \S(\Phi)$ and $\mathfrak{I} \subseteq \T(\graph(\Phi))\setminus \S(\Phi)$
are such that for all $T \in \T(\graph(\Phi))$ either  $F\subseteq T$ for some $F \in \mathfrak{F},$
or $T\subseteq I$ for some $I \in \mathfrak{I}.$   Then
\[
\H(\Phi) =  \dmax_{T \in \mathfrak{F}}
\|\Phi_T^{-1}\| =  \dmax_{T \in \mathfrak{F}}  \frac{1}
{\dmin_{u\in \Phi_T^*(v)\atop\|\Pi_{\Image(\Phi_T)}(v)\|^*=1}\|u\|^*}.
\]
\end{description}
\end{corollary}

The proof of Theorem~\ref{thm.main.gral} relies on the following technical lemma.

\begin{lemma}\label{lemma.rel.surj}
Let $\Phi:\R^n \rightrightarrows \R^m$ be a polyhedral set-valued mapping. Then
\[
\H(\Phi)=
\max_{T\in \S(\graph(\Phi))} \|\Phi_T^{-1}\| =\max_{T\in \T(\graph(\Phi))} \|\Phi_T^{-1}\|.
\]
\end{lemma}

\medskip

\begin{proof}[Proof of Theorem~\ref{thm.main.gral}]
To ease notation, throughout this proof let $\G := \graph(\Phi)$.
We will prove the following equivalent statement to~\eqref{eq.Hoffman.set}:
For all $b \in \Image(\Phi)$ and $(u,v)\in \G$
\begin{equation}\label{eq.Hoffman.bound}
\dist(u,\Phi^{-1}(b))\le \H(\Phi)\cdot\|b-v\|.
\end{equation}
Assume $b-v\ne 0$ as otherwise there is nothing to show.
We proceed by contradiction.   Suppose $b \in \Image(\Phi)$ and $(u,v) \in \G$ are such that $b-v\ne 0$ and
\begin{equation}\label{eq.contra}
\vertiii{x-u} > \H(\Phi) \cdot \|b-v\|
\end{equation}
for all $x$ such that $(x,b)\in \G$.
Let
$
d:= \frac{b-v}{\|b-v\|}
$
 and consider the optimization problem
\begin{equation}\label{eq.opt.prob}
\begin{array}{rl}
\displaystyle\max_{w,t} & t \\
\st& (u+w,v+td) \in \G, \\
& \|w\| \le \H(\Phi) \cdot t.
\end{array}
\end{equation}
Since $b\in \Image(\Phi)$ it follows that  
$d = (b-v)/\|b-v\| \in \Image(\Phi_{T})$ for $T:=T_{\G}(u,v) \in \T(\G)$.
Since $\|d\|=1$, the definition of $\|\Phi_{T}^{-1}\|$ and Lemma~\ref{lemma.rel.surj} imply that there exists $(z,d)
\in T$ with $\|z\|\le \|\Phi_{T}^{-1}\| \le \H(\Phi)$.  By the construction of $T = T_{\G}(u,v)$ it follows that for $t > 0$ sufficiently small
$(u+tz,v+td) \in \G$ and so $(w,t) := (tz,t)$ is feasible for problem~\eqref{eq.opt.prob}.  Let
$$S:=\{(w,t) \in \R^n \times \R_+: (w,t) \text{ is feasible for }~\eqref{eq.opt.prob} \}.$$
Assumption~\eqref{eq.contra} implies that $t < \|b-v\|$ for all $(w,t)\in S$.  In addition, since $\G$ is polyhedral,
it follows that $S$ is compact. Therefore~\eqref{eq.opt.prob} has an optimal solution $(\bar w,\bar t)$ with $0<\bar t
< \|b-v\|.$

Let $(u',v'):= (u + \bar w,v+\bar t d) \in \G$.
Consider the modification of~\eqref{eq.opt.prob} obtained by replacing $(u,v)$ with $(u',v')$, namely
\begin{equation}\label{eq.opt.prob.mod}
\begin{array}{rl}
\displaystyle\max_{w' ,t'} & t' \\
\st& (u'+w',v'+t'd) \in \G, \\
& \vertiii{w'} \le \H(\Phi)\cdot t'.
\end{array}
\end{equation}
Observe that
$
b - v' = b-v-\bar t d = (\|b-v\| - \bar t)d \ne 0.
$
Again since $b\in \Image(\Phi)$ it follows that $d= \frac{b-v'}{\|b-v'\|}  \in\Image(\Phi_{T'})$ for
$T':=T_{\G}(u',v')$.  Hence there exists $(z',d)\in T'$ such that  $\|z'\|\le \|\Phi_{T'}^{-1}\| \le \H(\Phi)$.
Therefore,~\eqref{eq.opt.prob.mod} has a feasible point $(w',t') = (t'z',t')$ with $t' > 0$.  In particular
$(u'+w',v'+t'd) = (u+\bar w + w', v + (\bar t + t')d) \in \G$ with $\|\bar w + w'\| \le \|\bar w\| + \|w'\| \le
\H(\Phi) \cdot(\bar t +t')$ and $\bar t + t' >\bar t$.  This contradicts the optimality of $(\bar w,\bar t)$
for~\eqref{eq.opt.prob}.

To show that the bound is tight, suppose $\H(\Phi) = \|\Phi_T^{-1}\| > 0$ for some $T\in \S(\Phi) \subseteq \T(\G)$.
The construction of $\vertiii{\Phi_T^{-1}}$ implies that there exists $d \in \R^m$ with  $\|d\|=1$ such that the
problem
\begin{equation}\label{eq.tight}
\begin{array}{rl}
\displaystyle\min_{z} & \vertiii{z} \\
\st & (z,d) \in T
\end{array}
\end{equation}
is feasible and has an optimal solution $\bar z$ with $\|\bar z\| = \|\Phi_T^{-1} \| = \H(\Phi )>0$.  Let $(u,v)\in
\G$ be such that $T = T_{\G}(u,v)$.   Let $b:=v+td$ where $t > 0$ is small enough so that $(u,v) + t(\bar z,d)\in \G$.
Observe that
$b \in \Image(\Phi)$ and $b - v = t d \ne 0$.
To finish, notice that if $x\in \Phi^{-1}(b)$ then $(x-u,b-v) = (x-u,td) \in T_{\G}(u,v) = T$.  The optimality of
$\bar z$ then implies that
\[
\|x-u\| \ge \H(\Phi) \cdot t = \H(\Phi)\cdot \|b-v\|.
\]
Since this holds for all $x\in \Phi^{-1}(b)$ and $b-v\ne 0$, it follows that
$\dist(u,\Phi^{-1}(b)) \ge \H(\Phi)\cdot \|b-v\| \ge \H(\Phi )\cdot\dist(b,\Phi(u))>0.$
\end{proof}

The proofs of Theorem~\ref{thm.main.gral.alt} and Lemma~\ref{lemma.rel.surj}  rely on the following  convex duality
construction.
Observe that  for a polyhedral convex cone $T\subseteq\R^n \times \R^m$
\[
\begin{array}{rl}
\|\Phi_T^{-1}\| = \dmax_{y} & \vertiii{\Phi_T^{-1}(y)}\\
\st &  y \in \Image(\Phi_T) \\
& \|y\|\le 1,
\end{array}
\]
where
\begin{equation}\label{primal.Hoffman}
\begin{array}{rl}
\vertiii{\Phi_T^{-1}(y)} := \dmin_{x} & \|x\| \\
\st & (x,y) \in T.
\end{array}
\end{equation}
By convex duality it follows that
\begin{equation}\label{dual.Hoffman}
\begin{array}{rl}
\vertiii{\Phi_T^{-1}(y)} =  \dmax_{u,v} & -\ip{v}{y}\\
\st & \|u\|^* \le 1 \\
& (u,v)\in T^*.
\end{array}
\end{equation}
Therefore when $T\subseteq\R^n\times \R^m$ is a polyhedral cone we have
\begin{equation}\label{eq.norm.Hoffman}
\begin{array}{rl}
\vertiii{\Phi_T^{-1}} = \dmax_{u,v,y} &  -\ip{v}{y} \\
\st  &  y \in \Image(\Phi_T)\\
& \|y\|\le 1 \\
& \|u\|^* \le 1\\
& (u,v)\in T^*.
\end{array}
\end{equation}

\begin{proof}[Proof of Theorem~\ref{thm.main.gral.alt}]
Let $T\in \S(\graph(\Phi))$.  Since  $\Phi_T$ is relatively surjective, from~\eqref{eq.norm.Hoffman} it follows that
\[
\begin{array}{rl}
\vertiii{\Phi_T^{-1}} = \dmax_{u,v} &  \|\Pi_{\Image(\Phi_T)}(v)\|^* \\
\st & \|u\|^* \le 1 \\
& u\in \Phi_T^*(v).
\end{array}
\]
The latter quantity is evidently the same as
\[
\dfrac{1}
{\dmin_{u\in \Phi_T^*(v)\atop\|\Pi_{\Image(\Phi_T)}(v)\|^*=1}\|u\|^*}.
\]
\end{proof}

Our proof of Lemma~\ref{lemma.rel.surj} relies on the following equivalence between {\em surjectivity} and {\em non-singularity} of sublinear mappings.
A standard convex separation argument shows that  a closed sublinear mapping $\Phi:\R^n \rightrightarrows \R^m$ is
surjective if and only if
\begin{equation}\label{eq.non.sing}
(0,v) \in \graph(\Phi)^* \Rightarrow v=0.
\end{equation}
Condition~\eqref{eq.non.sing} is a kind of {\em non-singularity} of  $\Phi^*$ as it can be rephrased as $0 \in
\Phi^*(v) \Rightarrow v=0.$

\begin{proof}[Proof of Lemma~\ref{lemma.rel.surj}]
Without loss of generality assume $\affine(\Image(\Phi)) = \R^m$ as otherwise we can work instead with the  modified
mapping   $\Phi_0 : \R^n \rightrightarrows L$ defined via
$
\Phi_0(x) := \Phi(x) - y_0,
$
where $y_0 \in \Image(\Phi)$ and $L$ is the lineality space of $\affine(\Image(\Phi))$, that is,
$L=\affine(\Image(\Phi))-y_0$.

To ease notation let $\G:=\graph(\Phi)$.  We need to show that
\[
\dmax_{T\in \T(\G)}  \|\Phi_T^{-1} \| =
\dmax_{T\in \S(\G)}  \|\Phi_T^{-1} \|.
\]
By construction, it is immediate that
\[
\dmax_{T\in \mathcal T(\G)}  \|\Phi_T^{-1} \| \ge
\dmax_{T\in \S(\G)}  \|\Phi_T^{-1} \|.
\]
To prove the reverse inequality let $T \in \mathcal T(\G)$ be fixed and let $(\bar u,\bar v, \bar y)$ attain the
optimal value $\|\Phi_T^{-1} \|$ in~\eqref{eq.norm.Hoffman}.
Let $\bar F$ be the minimal face of $T^*$ containing $(\bar u, \bar v)$ and $\bar T := \bar F^*  \in \mathcal T(T)
\subseteq \T(\G)$.   As we detail below, $(\bar u,\bar v, \bar y)$ can be chosen so that $\Phi_{\bar T}$ is
surjective.
Since $\|\bar y\|\le 1$ we have
\[
\|\Phi_T^{-1} \| = -\ip{\bar v}{ \bar y} \le \|\bar v\|^*.
\]
Furthermore, $(\bar u,\bar v) \in \bar F \subseteq T^* = \graph(\Phi_{\bar T})^*$ and $\|\bar u\|^* \le 1$, thus
Theorem~\ref{thm.main.gral.alt} yields
\[
\|\Phi_T^{-1} \| \le \|\bar v\|^* \le \|\Phi_{\bar T}^{-1} \|.
\]
Since this holds for any fixed $T \in\mathcal T(\G)$, it follows that
\[
\dmax_{T\in \mathcal T(\G)} \vertiii{\Phi_T^{-1} } \le  \dmax_{\bar T\in \S(\G)}  \|\Phi_{\bar T}^{-1} \|.
\]

It remains to show that $(\bar u,\bar v, \bar y)$ can be chosen so that $\Phi_{\bar T}$ is surjective, where $\bar T =
\bar F^*$ and $\bar F$ is the minimal face of $T^*$ containing $(\bar u,\bar v)$.  To that end, pick a
solution $(\bar u,\bar v,\bar y)$ to~\eqref{eq.norm.Hoffman} and consider the set
$$
V:=\{v \in \R^m: \ip{v}{\bar y} = \ip{\bar v}{\bar y}, \, (\bar u,v)\in T^*\}.$$
In other words, $V$ is the projection of the set of optimal solutions  to~\eqref{eq.norm.Hoffman} of the form $(\bar
u, v, \bar y)$.  Since $T$ is polyhedral, so is $T^*$ and thus $V$ is a polyhedron.
Furthermore, $V$ must have at least one extreme point.  Otherwise there exist $\hat v\in V$ and a nonzero $\tilde v\in
\R^m$ such that $\hat v + t\tilde v \in V$ for all $t \in \R$.   In particular, $(\bar u, \hat v + t\tilde v) \in T^*$
for all $t \in \R$ and thus  $(0,t \tilde v) \in T^*$ for all $t\in\R$.
The latter in turn implies $\Image(\Phi_T) = \{y\in \R^m: (x,y) \in T \text{ for some } x\in \R^n\} \subseteq \{y\in
\R^m: \ip{\tilde v}{y} =0\}$. Since $T\in \T(\G)$, if follows that $\Image(\Phi)) \subseteq\{y\in \R^m: \ip{\tilde
v}{y} =0\}$ and thus
$\affine(\Image(\Phi))\subseteq\{y\in \R^m: \ip{\tilde v}{y} =0\}$ thereby contradicting the assumption
$\affine(\Image(\Phi)) =\R^m$.  By replacing $\bar v$ if necessary, we can assume that
$\bar v$ is an extreme point of $V$. We claim that the minimal face $\bar F$ of $K^*$ containing $(\bar u,\bar v)$
satisfies
\[
(0,v') \in \bar F = \bar T^* \Rightarrow v' = 0
\]
thereby establishing the surjectivity of $\Phi_{\bar T}$ (cf., \eqref{eq.non.sing}).
To prove this claim, proceed by contradiction.  Assume $(0,v') \in \bar F$ for some nonzero $v'\in \R^m$.
The choice of $\bar F$ ensures that $(\bar u,\bar v)$ lies in the relative interior of $\bar F$ and thus for $t>0$
sufficiently small both $(\bar u,\bar v + tv')\in \bar F\subseteq T^*$ and
$(\bar u,\bar v - tv')\in \bar F\subseteq T^*$.  The optimality of $(\bar u,\bar v,\bar y)$  implies that both
$\ip{\bar v+tv'}{\bar y } \ge \ip{\bar v}{ \bar y}$ and $\ip{\bar v -tv'}{\bar y} \ge \ip{\bar v}{ \bar y}$ and so
$\ip{v'}{ \bar y} = 0$.  Thus both $\bar v + tv' \in V$ and $\bar v -tv'\in V$ with $tv'\ne 0$ thereby contradicting
the assumption that $\bar v$ is an extreme point of $V$.
\end{proof}

To conclude this section, we briefly describe how the approach to compute Hoffman constants via the covering property in Section~\ref{sec:certificates} extends to the general context of polyhedral set-valued mappings.    Suppose $\Phi:\R^n\rightrightarrows \R^m$ is a polyhedral set-valued mapping.  Corollary~\ref{corol.phi.sets} suggests the following algorithmic approach to
compute $\H(\Phi)$: Find  $\mathfrak{F}\subseteq \S(\Phi)$ and $\mathfrak{I} \subseteq \T(\graph(\Phi))\setminus \S(\Phi)$ that satisfy the following covering property:

\begin{quote}
For all $T \in \T(\graph(\Phi))$ either  $F\subseteq T$ for some $F \in \mathfrak{F},$
or $T\subseteq I$ for some $I \in \mathfrak{I}.$   
\end{quote}
Then compute
\[
\H(\Phi) =  \dmax_{T \in \mathfrak{F}}
\|\Phi_T^{-1}\| =  \dmax_{T \in \mathfrak{F}}  \frac{1}
{\dmin_{u\in \Phi_T^*(v)\atop\|\Pi_{\Image(\Phi_T)}(v)\|^*=1}\|u\|^*}.
\]

\section{Proofs of propositions in Section~\ref{sec.special}}
\label{sec.proof.props}
We next present the proofs of Proposition~\ref{prop.Hoffman.gral.2} and Proposition~\ref{prop.Hoffman.gral.alt.2}.  As
noted before, the other propositions in Section~\ref{sec.special} follow as special cases of these two results.

\medskip

Let $R\subseteq \R^n, \; A\in \R^{m\times n},$ and $C\in \R^{p\times n}$.
Construct $\Phi:\R^n \rightrightarrows \R^{m+p}$ as follows
\begin{equation}\label{eq.def.phi}
\Phi(x) = \left\{\begin{array}{ll} \{(Ax+s,Cx): s\ge 0\} & \text{ if } x\in R \\ \emptyset & \text{
otherwise.}\end{array}\right.
\end{equation}
Observe that $\Phi$ is polyhedral since $R$ is a polyhedron.

\begin{proof}[Proof of Proposition~\ref{prop.Hoffman.gral.2}] For $\Phi$  as in~\eqref{eq.def.phi} we have
\[
\S(\graph(\Phi)) = \{T_{J,K}: (J,K)\in \sJ(A,C\vert R)\}
\]
where
$ T_{J,K} = \{(x,Ax+s,Cx):x\in K, s_J \ge 0\}
$. Next, observe that for $(J,K)\in \sJ(A,C\vert R)$ we have $(y,w)\in \Phi_{T_{J,K}}(x) \Leftrightarrow x\in K, \, A_Jx\le y_J,$
and $Cx= w$.
Therefore $$\H(\Phi)=\max_{(J,K)\in \sJ(A,C\vert R)} \|\Phi_{T_{J,K}}^{-1}\| =
\dmax_{(J,K)\in \sJ(A,C\vert R)}
\dmax_{(y,w)\in  \R^m \times C(K)\atop \|(y,w)\|\le 1} \dmin_{x\in L \atop A_Jx \le  y_J, Cx = w} \|x\| = H(A,C \vert
R).
$$ Furthermore, $\dom(\Phi) = R$ and $\Image(\Phi) = \{(Ax+s,Cx): x\in R, \, s \ge 0\}$.  Therefore
Theorem~\ref{thm.main.gral} implies that for all $(b,d) \in \Image(\Phi) = \{(Ax+s,Cx): x\in R, \, s \ge 0\}$ and
$u\in \dom(\Phi) = R$
\begin{align*}
\dist(u,P_A(b)\cap C^{-1}(d) \cap R) &= \dist(u,\Phi^{-1}(b,d)) \\ &\le
\H(\Phi) \cdot \dist((b,d), \Phi(u)) \\&= H(A,C\vert R)  \cdot \dist((b-Au,d-Cu),  \R^m_+\times\{0\}).
\end{align*}
Theorem~\ref{thm.main.gral} also implies that this bound is tight.

\end{proof}
\begin{proof}[Proof of Proposition~\ref{prop.Hoffman.gral.alt.2}] Observe that for $(J,K)\in \sJ(A,C\vert R))$ and
$T:=T_{J,K}$ we have $\Image(\Phi_{T}) =  \R^m\times C(K)$ and $u\in \Phi_{T}^*(v,z)$ if and only if
$A\transp v + C\transp z - u \in K^*, \; v_J \ge 0,$ and $v_{J^c} = 0$.  Hence
\[
\dmin_{u \in \Phi_{T}^*(v,z)\atop \|\Pi_{\Image(\Phi_{T})(v,z)\|^* = 1}} \|u\|^* = \dmin_{v\in \R^J_+, z\in
C(K)\atop\|(v,z)\|^* = 1, A_J\transp v + C\transp z-u \in L^*}
\|u\|^*.
\]
To finish, apply Theorem~\ref{thm.main.gral.alt} and the facts  $\S(\graph(\Phi)) = \{T_{J,K}: (J,K)\in \sJ(A,C\vert
R)\}$ and $\H(\Phi) = H(A,C \vert R)$ established in the previous proof.

\end{proof}

\section{Conclusions}

We provide a characterization of the Hoffman constant for a system of linear inequalities and equations relative to a reference polyhedron (Proposition~\ref{prop.Hoffman.gral.2}).  Our characterization is stated as the largest of a finite collection of easily computable Hoffman constants (Proposition~\ref{prop.Hoffman.gral.alt.2}).  

We describe how our characterization can be leveraged to design two classes of algorithmic procedures to compute Hoffman constants.  One of them is based on a formulation of the Hoffman constant as a mathematical program with linear complementarity constraints (Proposition~\ref{prop:LCP}).  The other one is based on a certain type of covering property (Algorithm~\ref{alg:bb}).

We also develop the concept of Hoffman constant and generalize our characterization and covering property to compute it in the more general context of polyhedral set-valued mappings (Theorem~\ref{thm.main.gral} and Theorem~\ref{thm.main.gral.alt}).

\section*{Acknowledgements}

Javier Pe\~na's research has been funded by NSF grant CMMI-1534850.

\bibliographystyle{plain}


\end{document}